\documentclass[12pt]{amsart}
\usepackage{amssymb,latexsym,amsthm, amsmath, comment}
\usepackage[all]{xy}
\usepackage{graphicx, comment, eufrak}
\usepackage{fullpage}
\usepackage{color}

\ifx\pdfoutput\undefined
\usepackage{hyperref}
\else
\usepackage[pdftex,colorlinks=true,linkcolor=blue,urlcolor=blue]{hyperref}
\fi

\theoremstyle{plain}
\newtheorem{theorem}{Theorem}
\newtheorem{corollary}{Corollary}
\newtheorem{proposition}{Proposition}

\newtheorem*{ques}{Question}

\newtheorem{lemma}{Lemma}
\theoremstyle{definition}

\newtheorem{definition}{Definition}
\newtheorem{remark}{Remark}

\begin{document}

\title[Self affine Delone sets and deviation phenomena]
      {Self affine Delone sets and deviation phenomena}
\author{Scott Schmieding}
\address{University of Maryland}
\email{schmiedi@math.umd.edu}
\author{Rodrigo Trevi\~no}
\address{Courant Institute of Mathematical Sciences\\
         New York University}
\email{rodrigo@math.nyu.edu}
\date{\today}

\begin{abstract}
We study the growth of norms of ergodic integrals for the translation action on spaces coming from expansive, self-affine Delone sets. The linear map giving the self-affinity induces a renormalization map on the pattern space and we show that the rate of growth of ergodic integrals is controlled by the induced action of the renormalizing map on the cohomology of the pattern space up to boundary errors. We explore the consequences for the diffraction of such Delone sets, and explore in detail what the picture is for substitution tilings as well as for cut and project sets which are self-affine. We also explicitly compute some examples.
\end{abstract}

\maketitle

\section{Introduction}
The study of mathematical quasicrystals is dominated by the study of Delone sets and the many properties they possess, as well as objects constructed from them which have physical interpretation, such as their diffraction spectrum. Given a Delone set $\Lambda$, which are often taken to be mathematical models of quasicrystals, one can form a dynamical system given by the translation action by $\mathbb{R}^d$ and one can form a ``pattern space'' associated to it given by the closure of all the $\mathbb{R}^d$ translations under the appropriate topology. When $\Lambda$ has some finiteness properties, its pattern space $\Omega_\Lambda$ is a compact metric space whereon $\mathbb{R}^d$ acts by translations in a uniquely ergodic way. The pattern space $\Omega_\Lambda$ is a foliated space with leaves corresponding to $\mathbb{R}^d$ orbits of different patterns in $\Omega_\Lambda$, although it is not a manifold. By a result of Dworkin \cite{dworkin}, there is a deep connection between the ergodic properties of the dynamical system defined by translations of $\Lambda$, i.e., by a minimal $\mathbb{R}^d$ action on $\Omega_\Lambda$ and the diffraction spectrum of $\Lambda$, which models the outcome of a physical experiment consisting of exposing a material to radiation.

It is the diffraction point of view which made Schechtman's discovery of quasicrystals remarkable: their diffraction spectrum exhibited certain symmetries which would not be possible if the material had crystalline or repetitive atomic structure \cite{schechtman}. Thus their structure more closely resembles \emph{aperiodic} configurations yet repetitive and somewhat ordered. These are properties shared with some well-known aperiodic tilings of the plane, such as the Penrose tiling. Another feature of well known aperiodic tilings of the plane such as the Penrose, is their abundance of self-similarities. This is immediately seen when one recognizes such tilings as arising from a substitution rule.

In this paper we study the rate of convergence of ergodic integrals for systems which come from translations of Delone sets when the Delone sets are self-affine, which is slightly more general than self-similar. We obtain a theorem about the rates of fluctuations of ergodic integrals and show that the action on the cohomology of an associated topological space by the self-affinity controls the rates of growth of ergodic integrals. These results are even of physical significance by the aformentioned result of Dworkin. The aim of this paper is to explore and exploit the connection between ergodic theory, cohomology, and diffraction when a Delone set is self-affine.
\subsection{Statement of results}
\label{sec:results}
In this section we give sufficient background to state our results. A more detailed explanation of terms can be found in \S \ref{sec:background}.

Let $\Lambda$ be a countably infinite closed subset of $\mathbb{R}^d$ and denote by $B_r(x)$ the Euclidean ball of radius $r$ around $x\in \mathbb{R}^d$, or $B_r$ the ball around zero in $\mathbb{R}^d$. For any closed set we consider the translation action on $\mathbb{R}^d$: $\Lambda \mapsto \varphi_t(\Lambda)$, $t\in\mathbb{R}^d$. As such $x \in \varphi_t(\Lambda)$ if and only if $ (x -t) \in \Lambda.$

$\Lambda$ is \emph{uniformly discrete} if there exists a $r > 0$ such that for any two $x,y\in\Lambda$, $x\neq y$, we have that $|x-y| > r$. $\Lambda$ is \emph{relatively dense} if there exists an $R>0$ such that $B_R\cap  \varphi_t(\Lambda) \neq \varnothing$ for all $t\in \mathbb{R}^d$.

\begin{definition}
$\Lambda$ is a \emph{Delone set} if it is uniformly discrete and relatively dense in $\mathbb{R}^d$. A \emph{cluster} $P$ of $\Lambda$ is a non-empty finite subset $P\subset\Lambda$.
\end{definition}
Associated to $\Lambda$, there is a foliated compact metric space $\Omega_\Lambda$ called the \emph{pattern space} of $\Lambda$ on which $\mathbb{R}^d$ acts by translations. In this paper, we will always assume this action is minimal and uniquely ergodic. The space $\Omega_\Lambda$ contains $\Lambda$, all of its translations by $\mathbb{R}^d$ and usually many other other Delone sets (see \S \ref{sec:background} for the definition).
\begin{definition}
A Delone set $\Lambda$ is \emph{self-affine} if there exists an expanding matrix $A\in GL^+(d,\mathbb{R}) := \mathrm{exp} (\mathfrak{gl}(d,\mathbb{R}))$ and a measure preserving homeomorphism $\Phi_A:\Omega_\Lambda\rightarrow \Omega_\Lambda$ such that $A$ induces the conjugacy
\begin{equation}
\label{eqn:mph}
 \Phi_A\circ \varphi_{ t} = \varphi_{A t} \circ \Phi_A
\end{equation}
for any $t\in\mathbb{R}^d$. By expanding, we mean that $A$ has all eigenvalues of norm greater than 1.
\end{definition}
When $A$ is pure dilation ($\tau\cdot\mathrm{Id}$ for some $\tau>1$) then $A$ gives $\Lambda$ a \emph{self-similarity}. Otherwise, we say that $A$ gives $\Lambda$ \emph{self-affinity}.

We note that tilings with a substitution rule (see \S \ref{sec:subs} for a definition) naturally define self-affine Delone sets. More generally, substitution rules can also be defined for Delone sets without relying on tilings \cite{LagariasWang}.

The pattern space $\Omega_\Lambda$ is not a manifold but has a foliated topology in the sense of \cite{global}. As such, cohomology theories have been developed to study topological invariants of this space. The relevant theories are discussed in \S \ref{sec:cohom}, but for now we will denote the cohomology of $\Omega_\Lambda$ as $H^*(\Omega_\Lambda;\mathbb{C})$. We will only consider the cases where the cohomology space $H^k(\Omega_\Lambda;\mathbb{C})$ is trivial whenever $k\not\in\{0,\dots, d\}$.

\begin{definition}
A Delone set $\Lambda$ is \emph{renormalizable of finite type} (RFT) if it is self-affine and $\dim H^*(\Omega_\Lambda;\mathbb{C})<\infty$.
\end{definition}
The fact that $\Phi_{A}$ is measure-preserving, satisfies (\ref{eqn:mph}), together with the structure of the cylinder sets in the foliated topology (\ref{eqn:charts1}), equips $\Omega_{\Lambda}$ with the structure of a Smale space, in the sense of \cite{putnam:notes}. It is not hard to show that $\Omega_{\Lambda}$ is non-wandering under $\Phi_{A}$, in which case \cite[Theorem 4.4.2]{putnam:notes} implies that there is a point $\Lambda_{0} \in \Omega_{\Lambda}$ which is periodic under the map $\Phi_{A}$. If $\Lambda$ is RFT under the self-homeomorphism $\Phi_{A}$ satisfying (\ref{eqn:mph}), then it is also RFT under the self-homeomorphism $\Phi_{A}^{k}$ for any $k \in \mathbb{N}$. By replacing $\Phi_{A}$ with a large enough power, and noticing that $\Omega_{\Lambda_{0}} = \Omega_{\Lambda}$, \emph{we will assume throughout that $\Lambda$ is fixed under $\Phi_A$}. The condition that $\Phi_{A}$ fixes $\Lambda$ is first used in Lemma \ref{lem:indMap}.

We will list the eigenvalues $\lambda_i$ of $A$ by decreasing norm: $|\lambda_1|\geq \cdots\geq |\lambda_d| > 1$. By the conjugacy in (\ref{eqn:mph}), the homeomorphism $\Phi_A$ induces a non-trivial action on the cohomology of the pattern space:
$$\Phi^*_A: H^*(\Omega_{\Lambda};\mathbb{C})\longrightarrow H^*(\Omega_{\Lambda};\mathbb{C}).$$

For RFT Delone sets $\Lambda$, since the dimension of the cohomology $H^*(\Omega_{\Lambda};\mathbb{C})$ is finite dimensional, the action of $\Phi^*_A$ is defined by $d$ matrices $\mathcal{A}_i$ so that
$$\mathcal{A}_i: H^i(\Omega_{\Lambda};\mathbb{C}) \longrightarrow H^i(\Omega_{\Lambda};\mathbb{C})$$
is the induced action. In such case, we denote by $|\nu_1| > \dots> |\nu_r| > 0$ the norms of the $r$ distinct eigenvalues of the matrix $\mathcal{A}_d$. We will see that $\nu_1 = \det A$ which, without loss of generality (that is, by taking powers), we will assume it is always positive (see  (\ref{eqn:firstEig})).

Let $E_i$ be the generalized eigenspaces for the action of $\mathcal{A}_d$ on $H^d(\Omega_{\Lambda};\mathbb{C})$ induced by the map $\Phi_A$ corresponding to the eigenvalue $\nu_i$. The subspaces $E_i$ are decomposed as
$$E_i = \bigoplus_{j=1}^{\kappa(i)} E_{i,j},$$
where $\kappa(i)$ is the size of the largest Jordan block associated with $\nu_i$, as follows.
For each $i$, we choose a basis of classes $\{[\eta_{i,j,k}]\}$ with the property that $\langle [\eta_{i,j,1}],[\eta_{i,j,2}],\dots, [\eta_{i,j,s(i,j)}]\rangle = E_{i,j}$ and
\begin{equation}
\label{eqn:basisAct}
\mathcal{A}_d [\eta_{i,j,k}] = \left\{\begin{array}{ll}
\nu_i [\eta_{i,j,k}] + [\eta_{i,j-1,k}]   &\mbox{ for $j>1$,} \\
\nu_i [\eta_{i,j,k}]  &\mbox{ for $j=1$.}
  \end{array}\right.
\end{equation}
\begin{remark}
The choice of basis in (\ref{eqn:basisAct}) is not standard but it is done in order to make calculations in \S \ref{sec:erg} easier (in particular, those of Proposition \ref{cor:indAct}). The driving feature of this choice is that for any vector $v\in E_{i,j}$ we have that $\|\mathcal{A}_d v\|^n \leq C n^{j-1}|\nu_i|^n \|v\|$ for some $C$. In other words, all vectors in $E_{i,j}$ grow at the same rate under iterations of $\mathcal{A}_d$.
\end{remark}
Following the terminology of \cite{BufetovSolomyak}, we make the following definition.
\begin{definition}
The \emph{rapidly expanding subspace} $E^+(\Omega_\Lambda) \subset H^d(\Omega_\Lambda)$ is the direct sum of all generalized eigenspaces $E_i$ of $\mathcal{A}_d$ such that the corresponding eigenvalues $\nu_i$ of $\mathcal{A}_d$ satisfy
\begin{equation}
\label{eqn:RES}
\frac{\log |\nu_i|}{\log \nu_1} \geq 1 - \frac{\log|\lambda_d|}{\log \nu_1}.
\end{equation}
\end{definition}
\begin{remark}
We will see that in the case where $A$ is a uniform dilation (all $|\lambda_i|$ are equal) the criterion for belonging to $E^+(\Omega_\Lambda)$, by (\ref{eqn:firstEig}), reduces to
\begin{equation}
\label{eqn:RES2}
\frac{\log |\nu_i|}{\log \nu_1} \geq \frac{d-1}{d}.
\end{equation}
\end{remark}

We order the indices of distinct subspaces of $E^+(\Omega_\Lambda)$ as follows. First, we set $I^+_\Lambda = I^{+,>}_\Lambda \cup I^{+,=}_\Lambda$ be the index set of classes $[\eta_{i,j,k}]$ which form a generalized eigenbasis for $E^+(\Omega_\Lambda)$, where the indices in $I^{+,>}_\Lambda$ contains vectors corresponding to a strict inequality in (\ref{eqn:RES}) and the indices in $I^{+,=}_\Lambda$ correspond to vectors associated to eigenvalues which give an equality in (\ref{eqn:RES}). Note that $I^{+,=}_\Lambda$ can be empty but $I^{+,>}_\Lambda$ always has at least one element.
The set $I_\Lambda^+$ is partially ordered: $(i,j,k)\leq (i',j',k')$ if $L(i,j,T)T^{ds_i}\geq L(i',j',T)T^{ds_{i'}}$ for $T>1$, where
\begin{equation}
  \label{eqn:logs}
  L(i,j,T) = \left\{\begin{array}{ll}
  (\log T)^{j-1} &\mbox{ if $\nu_i$ satisfies (\ref{eqn:RES}) strictly} \\
  (\log T)^{j} &\mbox{ if $\nu_i$ satisfies equality in (\ref{eqn:RES})}
  \end{array}\right.
\end{equation}
and $s_i = \frac{\log |\nu_i|}{\log \nu_1}$. The order does not depend on the indices $k$.

Let $\Lambda$ be a RFT Delone set and $A\in GL^+(d,\mathbb{R}) = \mathrm{exp}( \mathfrak{gl}(d,\mathbb{R}))$ be the associated expanding matrix. Let $a\in \mathfrak{gl}(d,\mathbb{R})$ be the matrix which satisfies $\exp (a) =A$ and let $g_t = \exp(at)$. Given a good Lipschitz domain (defined in \S \ref{subsec:avSets}) $B_0$, we define the averaging family $\{B_T\}_{T\geq 0}$ by
\begin{equation}
\label{eqn:rescalledSets}
B_T =  g_{\sigma \log T} B_0,
\end{equation}
where $\sigma = d/\log\det A$. As such, we have that $\mathrm{Vol}(B_T)  = \mathrm{Vol}(B_0) T^d$.

The class of functions whose ergodic integrals we will study are called the \emph{transversally locally constant} functions. They are defined in \S \ref{sec:cohom}, are denoted by $C_{tlc}^\infty(\Omega_\Lambda)$, are dense in $L^1(\Omega_\Lambda)$, and can be thought of as functions which are $C^\infty$ smooth along the foliation and locally constant along the transverse direction in $\Omega_\Lambda$. Let $\rho = \mathrm{dim}\, E^+(\Omega_\Lambda)$.
\begin{theorem}[Deviations of ergodic averages]
\label{thm:deviations}
Let $\Lambda\subset\mathbb{R}^d$ be a RFT Delone set and $B_0\subset\mathbb{R}^d$ be a good Lipschitz domain. There exist a constant $C_{B_0,A}$ and $\rho$ $\mathbb{R}^d$-invariant distributions $\{\mathcal{D}_{i,j,k}\}_{(i,j,k)\in I^+_\Lambda}$ such that, for any $f \in C^\infty_{tlc}(\Omega_\Lambda)$, if there is an index $(i,j,k)$ such that $\mathcal{D}_{i',j',k'} (f) = 0$ for all $(i',j',k')<(i,j,k)$ but $\mathcal{D}_{i,j,k} (f) \neq 0$, then for $T>3$ and any $\Lambda_{0} \in \Omega_{\Lambda}$,
$$\left| \int_{B_T} f\circ \varphi_s(\Lambda)\, ds \right| \leq C_{B_0,A}  L(i,j,T)T^{d\frac{\log|\nu_{i}|}{\log \nu_1}  }\|f\|_{\infty}.$$
Moreover, there exists an $M_{B_0}$ such that if $\mathcal{D}_{i,j,k}(f) = 0$ for all $(i,j,k)\in I^+_\Lambda$ then
\begin{equation}
\label{eqn:bdryErr2}
\left| \int_{B_T} f\circ \varphi_s(\Lambda)\, ds \right| \leq M_{B_0} T^{d\left(1-\frac{\log|\lambda_d|}{\log\nu_1}\right)}\|f\|_{\infty}
\end{equation}
for all $T>0$. Finally, if $\mathcal{D}_{i,j,k}(f) = 0$ for all $(i,j,k)\in I_\Lambda^+$ and $B_0$ is a stellar time cube,
$$\lim_{T\rightarrow\infty}T^{-d\left(1-\frac{\log|\lambda_d|}{\log\nu_1}\right)} \int_{B_T}f    \circ \varphi_s(\Lambda)\, ds = 0.$$
\end{theorem}
Stellar time cubes are defined in \S \ref{subsec:avSets}.
\begin{remark}
We believe that through the approach of this paper, a similar statement can be derived for Delone sets with pattern spaces of infinite topological type, that is, spaces for which $H^*(\Omega_\Lambda;\mathbb{C})$ is not finite dimensional, as long as there is a homeomorphism of the form (\ref{eqn:mph}) and the spectrum of the induced action on cohomology has some disctrete components.
\end{remark}

\begin{remark}
  \label{rem:expansion1}
  The theorem will be a consequence of being able to express ergodic integrals through an ``expansion'' of the form
$$\int_{B_T} f\circ \varphi_s(\Lambda_0)\, ds = \sum_{(i,j,k)\in I^+_\Lambda} \mathcal{D}_{i,j,k}(f)\Psi_{i,j,k}^{B_0}(T)L(i,j,T)T^{d\frac{\log |\nu_{i}|}{\log\nu_1}}  + \mathcal{O}(|\partial B_T|),$$
for any function $f\in C_{tlc}^\infty(\Omega_\Lambda)$ and averaging sets of the form (\ref{eqn:rescalledSets}), where $\Psi_{i,j,k}^{B_0}:\mathbb{R}^+\rightarrow \mathbb{R}$ are bounded, continuous functions, dependent on $B_0$, which are responsible for fluctuations (see (\ref{eqn:intExpansion}) in Remark \ref{rmk:expansion}, \S \ref{subsubsec:speeds}). The distributions $\mathcal{D}_{i,j,k}$ do not depend on the set $B_0$, but one can think that they carry a ``homology'' class and their evaluation on $f\in C^{\infty}_{tlc}(\Omega_\Lambda)$ only depends on the cohomology class of $f$, which is defined as the $\Lambda$-equivariant cohomology class of the form $f(\star 1)$ (see \S \ref{sec:cohom}).
\end{remark}
\begin{remark}
The bounds in (\ref{eqn:bdryErr2}) are given by the boundary. In other words, the volume of the boundary of $B_T$ gives an error term of size $T^{d\left(1-\frac{\log|\lambda_d|}{\log\nu_1}\right)}$ (this is shown in \S \ref{subsec:avSets}). This is the so-called \emph{boundary effect}. The theorem says that fluctuations of ergodic integrals are described by eigenvalues of the induced action on cohomology which are large enough. Otherwise, the boundary of the averaging sets $B_T$ give a large error. When $A$ is pure dilation, the error caused by the boundary effects are of size $T^{d-1}$.
\end{remark}

\begin{remark}[Errors of convergence for patch frequencies]
\label{rmk:counting}
One can try to use Theorem \ref{thm:deviations} to calculate error terms of counting points of some Delone set $\Lambda$ inside the sets $B_T$. This can be done by looking at the ergodic integrals of a $C^\infty_{tlc}(\Omega_\Lambda)$ function $h$ which is supported on a small neighborhood of the canonical transversal. To know whether $\mathcal{D}_{i,j,k}(h)=0$ or not involves looking at the projection of an associated cohomology class $[\eta_h]\in H^d(\Omega_\Lambda)$ onto the different eigenvectors in (\ref{eqn:basisAct}). If these projections are non-zero, then the error terms are controlled by cohomology classes in the rapidly expanding subspace. If they are not, the geometry of the sets $B_T$, in particular, the boundary $\partial B_T$, plays a significant role. Our results about cut and project sets in \S \ref{subsec:renormCAPS} combined with Theorem \ref{thm:deviations}, which we state below as Theorem \ref{thm:cod1errs}, helps answer \cite[Question 55]{DF:limit} in the cases when a codimension-one cut and project set is RFT.

Questions about the rate of growth of counting functions, i.e., of asymptotic density of Delone sets, come up in the study of bounded displacement equivalence between Delone sets in $\mathbb{R}^d$ and $\mathbb{Z}^d$. See \cite{solomon:subs, ApCG:rapid, ApCG:rect, solomon:simple, HKW:separated} for related results.
\end{remark}
We also study the implications of Theorem \ref{thm:deviations} in the setting of substitutions as well as for some type of cut and project sets which are RFT (see \S \ref{subsec:renormCAPS} for a precise setup).
\begin{theorem}
\label{thm:cod1errs}
Let $\Lambda(K,\Gamma)$ be a codimension 1 almost canonical RFT cut and project set, $B_0$ a good Lipschitz domain and define $B_T$ as in (\ref{eqn:rescalledSets}). Then there exists a $C_{B_0}$ such that for any $f\in C^\infty_{tlc}(\Omega_{\Lambda(K,\Gamma)})$ and $T>1$,
$$\left| \int_{B_T} f\circ\varphi_t(\Lambda_0)\, dt - T^d\mu(f)\right|  \leq C_{B_0} \log(T) T^{d\left( 1-\frac{\log |\lambda_d|}{\log \nu_1}\right)} \|f\|_\infty.$$
\end{theorem}

The $\mathbb{R}^d$ action on pattern spaces $\Omega_\Lambda$ gives rise to an asymptotic cycle $\mathfrak{C}_\Lambda$. As in general cases, the asymptotic cycle in this setting is a topological invariant for ergodic $\mathbb{R}^d$-invariant probability measures on $\Omega_\Lambda$. It is called a cycle because it is a closed linear functional on pattern-equivariant $d$-forms, i.e., elements of the $\Lambda$-equivariant cohomology of $\Omega_\Lambda$. The asymptotic cycle defines an $\mathbb{R}^d$-invariant current and $\mathbb{R}^d$-invariant distribution on pattern equivariant functions. Invariant currents in the context of Delone sets were first investigated in \cite{Kellendonk-Putnam:RS}, where they defines \emph{the Ruelle-Sullivan map}, used to study cohomological properties of pattern spaces. The point of view here is different since it is motivated by the duality of asymptotic cycles with transverse invariant distributions and its connection to ergodic theory and diffraction. This point of view is influenced by the ergodic theory of translation flows on flat surfaces and, as such, the invariant currents defined here are inspired by Forni's basic currents in the study of deviation of ergodic averages of translations flows \cite{forni:deviation} (see \S \ref{subsec:related}).

We now briefly recall the main definitions used to define the diffraction measure.

Let $\Lambda$ be a Delone set with finite local complexity and uniform cluster frequency (see \S \ref{sec:background} for the necessary definitions). Given a translation-bounded measure $\upsilon$ on $\mathbb{R}^d$, let $\gamma(\upsilon)$ denote its autocorrelation, i.e.,
\begin{equation}
\label{eqn:autocorr}
\
\gamma(\upsilon) = \lim_{n\rightarrow\infty}\frac{1}{\mathrm{Vol}(F_n)} \left( \upsilon|_{F_n} * \tilde{\upsilon} |_{F_n}\right) = \lim_{n\rightarrow\infty} \frac{1}{\mathrm{Vol}(F_n)} \gamma_\Lambda^n ,
\end{equation}
where $\{F_n\}$ is some van Hove sequence used to average. For any set $\Lambda$, we denote its Dirac comb formally by
$$\upsilon_\Lambda = \sum_{x\in\Lambda} \delta_x.$$
For $\Lambda$ a Delone set with finite local complexity, $\upsilon_\Lambda$ has a unique autocorrelation given by the equation
\begin{equation}
\label{eqn:auto}
\gamma_\Lambda = \sum_{x,y\in\Lambda}\mathrm{freq}(x-y,\Lambda) \delta_{x-y},
\end{equation}
where $\mathrm{freq}(x-y,\Lambda)$ denotes the asymptotic frequency in $\Lambda$ of the cluster defined by two points $x,y\in \Lambda$ (see \S \ref{sec:background}). By Bochner's theorem the Fourier transform $\widehat{\gamma_\Lambda}=\sigma_\Lambda$ is a positive measure called the \emph{diffraction measure} for $\Lambda$. The measure $\widehat{\gamma_{\Lambda'}}$ is the same for any $\Lambda'\in\Omega_\Lambda$.

In this paper we show that the asymptotic cycle $\mathfrak{C}_\Lambda$ yields the autocorrelation measure of $\Lambda$, which is in fact the transverse invariant measure to the $\mathbb{R}^d$ action on $\Omega_\Lambda$. Therefore, as in the case of flat surfaces, the asymptotic cycle is nontrivial and dual to a transverse invariant measure.
\begin{theorem}
\label{thm:homology}
Let $\Lambda\subset \mathbb{R}^d$ be a Delone set with finite local complexity and uniform cluster frequency. Then there exists a unique asymptotic cycle $\mathfrak{C}_\Lambda$ whose homology class $[\mathfrak{C}_\Lambda]$ is nontrivial: there is a $d$-parameter family of $\Lambda$-equivariant cohomology classes $[\lambda^x]$, $x\in\mathbb{R}^d$, which are dual to $\mathfrak{C}_\Lambda$ in the sense that
$$ \widehat{\mathfrak{C}_\Lambda(\lambda^x)} =  \widehat{\gamma_{\Lambda}(x)} = \sigma_{\Lambda},$$
where $\gamma_\Lambda$ is the autocorrelation of $\Lambda$. The classes $[\lambda^x]$ are supported on the set $\Lambda-\Lambda$.
\end{theorem}
\begin{remark}
That $\gamma_{\Lambda}(x) = \widehat{\sigma_{\chi_\mho}}(x)$ is due to Dworkin \cite{dworkin}; what is new here is their relationship to the asymptotic cycle $\mathfrak{C}_\Lambda$.
\end{remark}
As observed in \cite{Kwapisz}, the existence of measure preserving homeomorphisms of type (\ref{eqn:mph}) reflects the existence of symmetries of the diffraction spectrum. Indeed, by \cite[Corollary 2]{lenz:continuity} we know that the pure point part of the diffraction measure is contained in the set of eigenvalues of the (uniquely ergodic) translation action of $\mathbb{R}^d$ on $\Omega_\Lambda$. Using the conjugacy (\ref{eqn:mph}) it can be shown that the set of eigenvalues is preserved by $A$, and thus that $A$ preserves the pure point part spectrum of the diffraction measure, that is, that it gives the diffraction measure a ``symmetry''.

Using Theorem \ref{thm:deviations} in conjunction with the interpretation of the diffraction measure as a current, we obtain terms for the convergence to the diffraction measure. In what follows, $\gamma_\Lambda^T$ is defined as in (\ref{eqn:autocorr}) but using the averaging sets (\ref{eqn:rescalledSets}).
\begin{theorem}[Deviations of diffraction measures]
\label{thm:DiffDevs}
Let $\Lambda$ be a RFT Delone set and $B_0$ a good Lipschitz domain. There exist $\rho$ distributions $\sigma_{i,j,k}$, supported on $\Lambda-\Lambda$ and indexed in the same way as the distributions in Theorem \ref{thm:deviations}, and a one-parameter family of distributions $\mathcal{B}^T_\Lambda$ such that
\begin{equation}
\label{eqn:diffDev}
\hat{\gamma}_\Lambda^T = \sum_{(i,j,k)\in I^+_\Lambda} L(i,j,T) \Psi_{i,j,k}^{B_0}(T)T^{d\frac{\log |\nu_{i}|}{\log\nu_1}} \sigma_{i,j,k} + \mathcal{B}_\Lambda^T
\end{equation}
in the sense of distributions, where $\sigma_{1,1,1} = \sigma_{\Lambda}$ the diffraction measure from Theorem \ref{thm:homology}, $T^{-d\left(1-\frac{\log|\lambda_{d}|}{\log\nu_1}+\varepsilon\right)}\mathcal{B}_\Lambda^T\longrightarrow 0$ as $T\rightarrow \infty$ for any $\varepsilon>0$, and the functions $\Psi_{i,j,k}^{B_0}:\mathbb{R}^+\rightarrow\mathbb{R}$ are bounded, continuous functions which depend on the set $B_0$ used to average.
\end{theorem}
\begin{remark}
By the isomorphism $i_\Lambda$ between $\Lambda$-equivariant functions in $\mathbb{R}^d$ and $C_{tlc}^\infty(\Omega_\Lambda)$ (see Theorem \ref{thm:tlc} in \S \ref{subsec:PEC}) the distributions $\sigma_{i,j,k}$ in Theorem \ref{thm:DiffDevs} give $\mathbb{R}^d$-invariant distributions on $\Omega_\Lambda$ associated to different homology classes. Because $\Omega_\Lambda$ has a foliated topology (see \S \ref{sec:background}), they are locally of the form $(i^{-1}_\Lambda)_*\sigma_{i,j,k}\times \mathrm{Leb}$.
\end{remark}
In well-known examples of self-affine Delone sets such as the vertex set of the Penrose tiling or Ammann-Beenker tiling (which are both substitution tilings), the rapidly expanding subspace has dimension greater than 1. We cover these examples in \S \ref{sec:ex}. However, in these examples, the relevant (non-leading) eigenvalues which satisfy (\ref{eqn:RES}) only do so as equalities (that is, $|I^{+,>}_\Lambda|=1$ while $|I^{+,=}_\Lambda|>0$). Following a suggestion of L. Sadun (for which we are grateful) we show in \S \ref{subsec:largeRES} how, using products, one can construct examples of self-similar substitutions with arbitrarily large rapidly expanding subspaces which satisfy (\ref{eqn:RES}) strictly (that is, $|I^{+,>}_\Lambda|>1$). These eigenvalues control the deviations of ergodic integrals at rates which are higher than the boundary effects. However, we only know of these types of constructions to generate such examples. Thus we end with the following.
\begin{ques}
Given any $k,d\in\mathbb{N}$, is there a self-affine Delone set $\Lambda\subset \mathbb{R}^d$ which does not come from a product such that at least $k$ eigenvalues of the induced action on $H^d(\Omega_\Lambda;\mathbb{C})$  satisfy (\ref{eqn:RES}) strictly?
\end{ques}
\subsection{Context}
\label{subsec:related}
The idea of an asymptotic cycle starts with the work of Schwartzman \cite{schwartzman:cycle}, where he defines the cycle as a topological invariant of an ergodic invariant probability measure for a flow on a manifold, somewhat generalizing the concept of rotation vector. This concept was eventually generalized to higher rank actions \cite{schwartzman:cycle2} on manifolds, where the asymptotic cycles are higher dimensional and still topological invariants of ergodic invariant probability measures. In the context of tilings and Delone sets, asymptotic cycles have been already considered in \cite{Kellendonk-Putnam:RS}. In the study of translation flows on flat surfaces, the asymptotic cycle plays an important role: for minimal translation flows on flat surfaces, the Hodge representative of the Poincar\'{e} dual of the asymptotic cycle defines a transverse invariant probability measure to the flow. When the flow is uniquely ergodic (which is generically true), this is the only transverse invariant probability measure.

More generally, Zorich discovered that there are other asymptotic cycles which are responsible for different rates of growth of ergodic integrals: his work first explored the relationship between the rates of deviations of Birkhoff sums over interval exchange transformations and quantitative properties of an associated hyperbolic system (see \cite{zorich-leaves}). This relationship was later conjectured \cite{kontsevich:hodge} and proved \cite{forni:deviation} to hold for ergodic integrals of smooth functions over translation flows on flat surfaces. More specifically, it was shown that the Lyapunov exponents of Kontsevich-Zorich cocycle
dictate the rate of deviations of ergodic averages of smooth functions over translation flows. This yielded the Zorich-Forni phenomenon: a hierarchy of spaces of (co)homology classes in bijection with a hierarchy of spaces of functions for which the growth of ergodic integrals is dictated by the quantitative properties of a renormalizing hyperbolic dynamical system (defined on a cohomological space). The Kontsevich-Zorich cocycle, in this situation, defines a renormalization dynamical system for the translation flow (see \cite{FM:intro} for an introduction). Thus, there is motivation to study the growth exponents for renormalizing dynamical systems since it can yield detailed information of related systems, which are sometimes even of physical interest. For example, the rates dictated by the Lyapunov exponents of the Kontsevich-Zorich cocycle have been shown to control the limiting large-scale geometry of associated systems, some of which come from physical models (see e.g. \cite{DHL:wind-tree, FSU:billiards}). We consider Theorem \ref{thm:DiffDevs} to be of this type.

Results concerning the speed of convergence of ergodic integrals of systems defined through Delone sets exist in the literature. Estimates for discrepancies for one dimensional self-similar tilings were studied by Adamczewski in \cite{Adamczewski} where the results depend on the eigenvalues of the substitution matrix through an inequality of the form (\ref{eqn:RES2}). These are estimates for deviations of ergodic integrals.

For higher dimensional Delone sets (which define higher rank actions), a manifestation of the Zorich-Forni phenomenon has already been hinted at in the case of self-similar tilings: in \cite{sadun:exact}, using \v{C}ech cohomology, Sadun shows that the rates of convergence of patch frequencies for self-similar substitution tilings are controlled by the top eigenvalues of the substitution matrix, up to boundary effects. The estimates there give rates of growth of ergodic integrals for functions which are related to the functions we consider here.

Likewise, the work of Bufetov and Solomyak \cite{BufetovSolomyak} addresses the deviation of ergodic averages for systems coming from self-similar substitution tilings and derive limit theorems from this. There are many parallels between our results here and those of \cite{BufetovSolomyak}, but we first note the significant differences. The methods of \cite{BufetovSolomyak} do not use cohomology and are based on earlier work of Bufetov on finitely-additive measures \cite{bufetov:cocycles}. Our approach here using cohomology allows us to apply our formalism to problems which come from mathematical physics such as diffraction and the spectral theory of random Schr\"odinger operators (see the last paragraph in this subsection). In addition, our results apply to self-affine systems, not just self-similar systems as in \cite{BufetovSolomyak}. Moreover, in \S \ref{subsec:renormCAPS}, we show how our results can be applied to self-affine cut-and-project sets which falls beyond the scope of self-similar tilings.

But the parallels between this work and \cite{BufetovSolomyak} are many. For example: 1) their results depend on an inequality of the form (\ref{eqn:RES2}) for eigenvalues of the substitution matrix since boundary effects are also present; 2) the class of functions they consider is different from the class of functions we consider. However, what is common in both cases is that the functions are transversally locally constant. The class of functions used in \cite{BufetovSolomyak} is called ``cylindrical functions'', which need less regularity compared to the ones used here but which are not dense in $L^2$ unlike the functions $C_{tlc}^\infty(\Omega_\Lambda)$ considered here; 3) the finitely-additive measures on the transversals of \cite{BufetovSolomyak} play the same role that the $\Lambda$-equivariant distributions and currents play here. In fact, one expects that these functionals are different manifestations of the same functionals.

Our results on deviations of ergodic integrals (e.g. Theorem \ref{thm:deviations}) are in the spirit of the ones in \cite{sadun:exact, BufetovSolomyak}. Indeed, we prove in Proposition \ref{prop:specCont} in \S \ref{sec:subs} that the spectrum for the induced action on cohomology is contained in the relevant spectrum in substitutions. In all the examples that we have worked out we have not seen a case where their relevant spectrum is different from ours. Thus, in such cases, our theorems recover theirs and apply to a more general class of systems.

The results of Sadun \cite{sadun:exact} and Bufetov-Solomyak \cite{BufetovSolomyak} are the first examples of phenomena similar to the Zorich-Forni phenomenon for higher rank actions of Abelian groups. Outside the context of tilings and Delone sets, the work of Cosentino and Flaminio \cite{CF:heisenberg} have also shown the relationship between cohomology and growth properties of norms of ergodic integrals for higher rank Abelian actions on Heisenberg manifolds, where the renormalization dynamics do not come from a self similarity necessarily, but a renormalizing flow on a homogeneous space.


We close by mentioning that we have found an application of the Theorem \ref{thm:deviations} to the spectral theory of random Schr\"odinger operators. More specifically, the asymptotic cycles here give rise to traces on a dense subalgebra of random self-adjoint operators on $\ell^2(\Lambda)$. In \cite{ST:traces}, we show that these traces control the error rates of convergence of the integrated density of states in Shubin's formula.

\subsection{Organization}
This paper is organized as follows. In \S \ref{sec:background} we review the basic notions related to Delone sets, pattern spaces and special constructions such as substitution systems and cut and project sets. In \S \ref{sec:cohom} we review the relevant cohomology theories which we will use. Sections \S \ref{sec:renorm} and \S \ref{sec:erg} deal with the renormalization procedure created by the homeomorphism related to RFT Delone sets, and how it relates to ergodic integrals and their growth. In \S \ref{sec:diff} we study the consequences on the diffraction spectrum. In \S \ref{sec:subs} we relate the renormalization given by substitution systems to the action on cohomology of the pattern space. In \S \ref{subsec:renormCAPS} we define and study certain RFT cut and project sets and apply the results on deviations of ergodic inegrals from \S \ref{sec:erg}. We finish by computing specific examples in \S \ref{sec:ex} and relate them to the results of \S \ref{sec:erg}.
\subsection{Acknowledgements}
R.T. wishes to thank J. Kellendonk for patiently explaining cohomology of pattern spaces, to G. Forni for many insightful discussions about deviations of ergodic averages for translation actions, and to B. Weiss for many useful conversations about tilings and Delone sets as well as many useful comments about an early draft of this paper which greatly improved it. R.T. was partially supported by the NSF under Award No. DMS-1204008, BSF Grant 2010428, and ERC Starting Grant DLGAPS 279893. A significant part of this was written while R.T. was visiting IMPA during the Programa de Pos-Doutorado de Verao 2015, and is grateful for the hospitality. This research was also supported in part by the National Science Foundation grant "RTG: Analysis on manifolds" at Northwestern University. We would also like to thank the participants of Arbeitsgemeinschaft: Mathematical Quasicrystals held at Mathematisches Forschungsinstitut Oberwolfach in October 2015 for many useful conversations. We are also very grateful to Lorenzo Sadun for pointing out a mistake in an early version of this paper. Finally, we are grateful to an anonymous referee for many helpful suggestions which improved the paper.

\section{Delone sets and pattern spaces}
\label{sec:background}
For a Delone set $\Lambda$, let $r_{min}(\Lambda)$ and $r_{max}(\Lambda)$ denote, respectively, the supremum of all $r>0$ and infimum of all $R>0$ which allow $\Lambda$ to satisfy the definitions of uniform discreteness and relative density. We call a Delone set $\Lambda$ \emph{aperiodic} whenever $\varphi_t(\Lambda) = \Lambda$ if and only if $t = 0$.

\begin{definition}
A Delone set $\Lambda$ has \emph{finite local complexity} (FLC) if for every $R>0$ there exists a finite set $Y\subset\Lambda$ such that for all $x\in \Lambda$ there exists a $y\in Y$ such that
$$B_R(x)\cap \Lambda = \varphi_{x-y}(B_R(y)\cap \Lambda).$$
\end{definition}
In other words, a Delone set $\Lambda$ has finite local complexity if for every $R>0$ there are finitely many clusters (up to translation) in any ball of radius $R$ centered at any point in $\mathbb{R}^d$. In this paper we will always consider Delone sets of finite local complexity.

Recall that there is a translation action $\varphi:\Lambda \mapsto \varphi_t(\Lambda) = \Lambda- t$ of the Delone set $\Lambda$ for any $t\in \mathbb{R}^d$. For any $\Lambda_1,\Lambda_2 \in \{\varphi_t(\Lambda)\}_{t\in\mathbb{R}^d}$ which are translates of $\Lambda$, we define the
\emph{local metric}
\begin{equation}
\label{eqn:metric}
\bar{d}(\Lambda_1,\Lambda_2) = \inf \{\varepsilon > 0: \mbox{there exist } x,y\in B_\varepsilon \mbox{ such that } B_{\varepsilon^{-1}}\cap \varphi_x(\Lambda_1) = B_{\varepsilon^{-1}}\cap \varphi_y(\Lambda_2) \}
\end{equation}
and define the \emph{local or pattern metric} to be
\begin{equation}
\label{eqn:metricChab}
d(\Lambda_1,\Lambda_2) = \min(2^{-\frac{1}{2}}, \bar{d}(\Lambda_1,\Lambda_2)).
\end{equation}
That this is a metric is proved in \cite{LMS02}. Consider the orbit closure of $\Lambda$ under the translation action:
$$\Omega_{\Lambda} \equiv \overline{\{ \varphi_t( \Lambda) : t\in \mathbb{R}^d\}},$$
where the closure is with respect to the metric given in (\ref{eqn:metricChab}). Since $\Lambda$ has finite local complexity, $\Omega_\Lambda$ is a compact metric space (see \cite[\S 5.4]{BG:book1}). We call $\Omega_\Lambda$ the \emph{pattern space of $\Lambda$} and the dynamical system the \emph{pattern dynamical system} which we denote by $(\Omega_\Lambda,\mathbb{R}^d)$.
\begin{definition}
A Delone set $\Lambda$ is \emph{repetitive} if for every compact set $K\subset \mathbb{R}^d$ there is a compact subset $K'\subset\mathbb{R}^d$ such that for every $t,t'\in\mathbb{R}^d$ there is a $s\in K'$ with $\varphi_t(\Lambda)\cap K = \varphi_{s+t'}(\Lambda)\cap K$.
\end{definition}
Repetitivity for Delone sets of finite local complexity is equivalent to minimality of the associated dynamical system $(\Omega_\Lambda,\mathbb{R}^d)$, i.e., every orbit is dense in $\Omega_\Lambda$ (see \cite[Proposition 5.4]{BG:book1}). If $\Lambda$ is aperiodic and repetitive, then the system is aperiodic in the sense that $\Lambda'$ is aperiodic for any $\Lambda'\in\Omega_\Lambda$.

The \emph{canonical transversal} of $\Omega_\Lambda$ is the set
$$\mho_\Lambda = \{ \Lambda' \in \Omega_\Lambda : \bar{0}\in\Lambda'\},$$
where $\bar{0}$ denotes the origin. For self-affine RFT Delone sets, from (\ref{eqn:mph}), it can be seen that the canonical transversal is mapped into itself by the homeomorphism $\Phi_A$.

The pattern space $\Omega_\Lambda$ is a foliated topological space (see, for example, \cite{KellendonkPEF}): the topology generated by the metric (\ref{eqn:metricChab}) can also be seen to be generated by \emph{cylinder sets} centered at some point $\Lambda_0\in\Omega_\Lambda$ of the form
\begin{equation}
\label{eqn:charts1}
\mathcal{C}_{\Lambda_0,\varepsilon, V} = \{\varphi_t(\Lambda') \in \Omega_\Lambda : B_{\varepsilon^{-1}}\cap \Lambda' = B_{\varepsilon^{-1}}\cap \Lambda_0, t\in V\}
\end{equation}
for $\Lambda_0 \in \Omega_\Lambda,\varepsilon < r_{min}(\Lambda)/2,$ and open sets $V\subset \mathbb{R}^d$ of diameter less than $r_{min}$. Such sets are the homeomorphic image of
\begin{equation}
\label{eqn:charts2}
\{\Lambda' \in \Omega_\Lambda: B_{\varepsilon^{-1}}\cap \Lambda' = B_{\varepsilon^{-1}}\cap \Lambda_0\} \times V
\end{equation}
under the map $\phi_{\Lambda_0,\varepsilon, V}^{-1}:(\Lambda',t)\mapsto \varphi_t(\Lambda')$. Pairs of the form $(\mathcal{C}_{\Lambda_0,\varepsilon,V},\phi_{\Lambda_0,\varepsilon, V})$ are the foliated charts of $\Omega_\Lambda$.

For a non-empty cluster $P$ of a Delone set $\Lambda$ and a bounded set $A\subset \mathbb{R}^d$ denote
$$L_{P}(A) \equiv \mbox{card}\{ t\in \mathbb{R}^d: \varphi_t(P) \subset A\cap\Lambda \},$$
which is the number of translates of $P$ contained in $A$.

For a bounded, measurable subset $F$ of $\mathbb{R}^d$, let
$$\partial^rF = \{x\in \mathbb{R}^d: \mbox{dist}(x,\partial F)\leq r\}.$$
A \emph{van Hove sequence} $\{F_i\}_{i>0}$ is a countably infinite collection of bounded measurable subsets of $\mathbb{R}^d$ satisfying
$$\lim_{n\rightarrow \infty} \frac{\mbox{Vol}( \partial^r F_n)}{\mbox{Vol}(F_n)} = 0$$
for all $r>0$.

\begin{definition}
Let $F_n$ be a van Hove sequence. A Delone set $\Lambda$ has \emph{uniform cluster frequencies} (UCF) (relative to $F_n$) if for any cluster $P$, the limit
$$\mbox{freq}(P,\Lambda) = \lim_{n\rightarrow\infty} \frac{L_{P}(\varphi_t( F_n))}{\mbox{Vol}(F_n)}$$
exists uniformly in $t\in\mathbb{R}^d$.
\end{definition}

\begin{proposition}[\cite{LMS02}]
\label{prop:UE}
Let $\Lambda$ be a Delone set with finite local complexity. Then the dynamical system $(\Omega_\Lambda, \mathbb{R}^d)$ is uniquely ergodic if and only if $\Lambda$ has uniform cluster frequency with respect to any van Hove sequence.
\end{proposition}
The unique invariant measure for systems defined by Delone sets of finite local complexity has a nice product form compatible with the foliated charts (\ref{eqn:charts1}) and (\ref{eqn:charts2}). The following is a restatement of \cite[Corollary 2.8]{LMS02} adapted to our setup here.
\begin{proposition}
\label{prop:measure}
Let $\Lambda_0 \in \Omega_\Lambda$ be a Delone set with finite local complexity and uniform cluster frequency and denote by $\mu$ the unique invariant measure for the $\mathbb{R}^d$ action on $\Omega_\Lambda$. Then there exists a $K(d)$ such that for any $\varepsilon<r_{min}(\Lambda_0)/2$,
$$\mu(\mathcal{C}_{\Lambda_0,\varepsilon, B_\varepsilon}) = \mathrm{Vol}(B_\varepsilon(0)) \cdot \mathrm{freq}(P_{\Lambda_0,\varepsilon},\Lambda_0) = K(d)\varepsilon^d \cdot \mathrm{freq}(P_{\Lambda_0, \varepsilon},\Lambda_0),$$
where $P_{\Lambda_0, \varepsilon} = B_{\varepsilon^{-1}}\cap \Lambda_0$.
\end{proposition}
This proposition shows that since the foliated topology of $\Omega_\Lambda$ is locally the product of a Cantor set and $B_1\subset \mathbb{R}^d$, the unique invariant measure for the $\mathbb{R}^d$ action also admits a local product structure, in terms of a transverse invariant measure (in the sense of \cite[Ch. IV]{global}) and Lebesgue measure. Indeed, since subsets of the transversal are given by clusters, by Proposition \ref{prop:measure}, for any cluster $P$, the unique transverse $\mathbb{R}^d$-invariant measure $\tau_\Lambda$ on $\mho_\Lambda$ satisfies
\begin{equation}
\label{eqn:treasure}
\tau_\Lambda(P) = \mathrm{freq}(P,\Lambda).
\end{equation}
\subsection{Cut and project sets}
\label{subsec:CAPS}
A special class of Delone sets are those given by a \emph{cut and project} scheme. In this subsection we describe their construction.

Consider two Euclidean spaces $E^\parallel$ and $E^\perp$ of dimensions $d$ and $n-d$, respectively, and identify $\mathbb{R}^n = E^\parallel\oplus E^\perp$. Denote by $\pi_\parallel:\mathbb{R}^n\rightarrow E^\parallel$ and $\pi_\perp: \mathbb{R}^n\rightarrow E^\perp$ the corresponding projections. Let $\Gamma \subset \mathbb{R}^n$ be a lattice of unit covolume, which can be identified with an element of $X_n = SL(n,\mathbb{R})/SL(n,\mathbb{Z})$. We always assume that $\Gamma$ is in \emph{totally irrational position with respect to $E^\parallel$ and $E^\perp$}, that is, that the maps $\pi_\parallel|_{\Gamma}$ and $\pi_\perp|_{\Gamma}$ are injective and have a dense image in $E^\parallel$ and $E^\perp$, respectively. Let $K\subset E^\perp$ be a compact set which is the closure of its non-empty interior and such that $\mathrm{Vol}(\partial K) = 0$.

\begin{definition}
The \emph{cut and project set} (CAPS) $\Lambda(\Gamma,K)\subset E^\parallel$ associated to $\Gamma $ and $K$ is the set
$$ \Lambda(\Gamma,K) = \{\pi_\parallel(z):z\in\Gamma\mbox{ and }\pi_\perp(z)\in K\}.$$
\end{definition}
That CAPS are Delone sets is well known. Under the assumptions stated above for the definition, they define a uniquely ergodic dynamical system $(\Omega_{\Lambda(K,\Gamma)}, \mathbb{R}^d)$; in the case $K$ also satisfies $\pi_{\perp}(\Gamma) \cap \partial K = \varnothing$, the associated dynamical system is also minimal. See \cite[\S 5]{RobinsonHalmos} for details.
\subsection{Substitution systems}
\label{subsec:subs}
Let $\mathfrak{T} = \{T^1,\dots, T^r\}$ be a disjoint collection of compact subsets of $\mathbb{R}^{d}$, called \emph{prototiles}. The set $\mathfrak{T}$ \emph{tiles} $\mathbb{R}^d$ if $\mathbb{R}^d$ can be expressed as a union $\bigcup_i T_i$ where each $T_i$ is a translation of one of the prototiles $T^j$ and the $T_{i}$ have disjoint interiors. We will always assume our prototiles are polyhedra, and that the tiles meet full face to full face in each dimension. This gives rise to a \emph{tiling} $\mathcal{T}$ of $\mathbb{R}^d$. The elements $T_i$ of a tiling are called \emph{tiles} of $\mathcal{T}$. We note that the vertex set of a tiling, $\Lambda_{\mathcal{T}}$, is a Delone set. A \emph{patch} of $\mathcal{T}$ is a finite collection of tiles $T_i$ of $\mathcal{T}$. For a tiling $\mathcal{T}$ and $r \ge 0$, let $B_{r}[\mathcal{T}] = \{T_{i} \in \mathcal{T} : T_{i} \cap \overline{B_{r}(0)} \not= \varnothing \}$. We denote the set of patches of $\mathcal{T}$ by $\mathcal{P}_\mathcal{T}$, and let $\Omega_{\mathfrak{T}}$ denote the set of tilings generated by $\mathfrak{T}$ satisfying the conditions outlined above. The tiling metric, analogous to the pattern metric defined in (\ref{eqn:metricChab}) (see \cite[1.2]{sadun:book} for a precise definition), makes $\Omega_{\mathfrak{T}}$ into a metric space. Given $\mathfrak{T}$, it is not always true that $\Omega_{\mathfrak{T}}$ is non-empty, but the substitution process described in Definition \ref{def:sub} below gives a classical method for guaranteeing this.\\
\begin{remark}
\label{rmk:backandforth}
There are several ways to obtain a tiling from a Delone set and vice versa. For example, if one has a tiling $\mathcal{T}$, one can pick the center of mass of every tile, and this choice for every tile in $\mathcal{T}$ gives a Delone set $\Lambda_\mathcal{T}$. This is one of many ways of assigning a point in every tile in order to obtain a Delone set. In the other direction, given a Delone set $\Lambda$, one can consider the Voronoi construction or a Delaunay partition to obtain tilings $\mathcal{T}_\Lambda$ which are built from $\Lambda$. In these constructions, properties such as finite local complexity, repetitivity, etc, carry over to the newly constructed object.
\end{remark}
\begin{definition}\label{def:sub}
Let $\phi \in GL^+(d,\mathbb{R})$ be an expanding linear map. A map $S:\mathfrak{T}\rightarrow\mathcal{P}_\mathcal{T}$ is called a \emph{tile-substitution} with expansion matrix $\phi$ if $S(T^i)$ is translation-equivalent to $\phi \cdot T^i$ for all $i$.
\end{definition}
In other words, every expanded prototile $\phi \cdot T^{i}$ can be decomposed into a patch made up of finitely many prototiles. The \emph{substitution matrix} (or \emph{incidence matrix}) associated to $S$ and $\phi$ is the $r\times r$ matrix $M_{S} = (m_{i,j})$ given by defining $m_{i,j}$ to be the number of translated copies of $T^{i}$ that appear in the patch $S(T^{j})$. Recall a matrix $A$ is called \emph{primitive} if there exists $l > 0$ such that $A^{l}$ has all positive entries. If the incidence matrix associated to a substitution is primitive we will say the substitution is primitive, and we will assume from here on that our substitutions are primitive.

The map $S$ can be extended to patches in a natural way, and thus to tilings: given a patch $P$, $\phi(P)$ is a new patch which may be subdivided using $S$ (we refer the reader to \cite[Def. 2.5]{solomyak:SS} for more details). A patch $P$ is \emph{allowed} (for $S$) if there exists a prototile $T^{j} \in \mathfrak{T}$ and $l \ge 0$ such that $P$ is contained in a translate of a patch of $S^{l}(T^{j})$. We define the \emph{tiling space} $\Omega_S$ associated to $S$ to be the set $\{\mathcal{T} \in \Omega_{\mathfrak{T}} : B_{r}[\mathcal{T}] $ is an allowed patch for $S$ for all $r \ge 0\}$. Given a tiling $\mathcal{T} \in \Omega_{S}$, the inflated tiling $\phi(\mathcal{T})$ may be subdivided using the substitution $S$ to produce a new tiling in $\Omega_{S}$, giving a continuous map $\Phi_S:\Omega_{S} \to \Omega_{S}$ (see the discussion following Definition 2.5 in \cite{solomyak:SS} for more details). Furthermore, translation gives an action of $\mathbb{R}^{d}$ on $\Omega_{S}$. It is well known (see \cite[Theorem 1.5]{sadun:book} for example) that if $S$ is a primitive substitution and $\mathcal{T} \in \Omega_{S}$ then the orbit of $T$ under $\mathbb{R}^{d}$ is dense in $\Omega_{S}$, and hence $\Omega_{S}$ agrees with the standard hull $\Omega_{\mathcal{T}}$.\\

Along with primitivity of $M_{S}$, we will also assume the following standard properties of $\Omega_{S}$:
\begin{itemize}
\item
$\Omega_{S}$ has finite local complexity, so each tiling $\mathcal{T} \in \Omega_{S}$ satisfies finite local complexity: for each $r \ge 0$ there are only finitely many patches (up to translation) of the form $B_{r}[\mathcal{T}]$.
\item
$\Omega_{S}$ is translation aperiodic, meaning for all $\mathcal{T} \in \Omega_{S}$, $\varphi_{t}(\mathcal{T}) = \mathcal{T}$ implies $t = \overline{0}$.
\end{itemize}

With these assumptions, the associated tiling space $\Omega_{S}$ is compact, and the action of $\mathbb{R}^{d}$ on $\Omega_{S}$ is minimal. The actions are related via the identity $\Phi_S(\varphi_{t}(\mathcal{T})) = \varphi_{\phi t}S(\mathcal{T})$. Furthermore, we have the following two facts, which can be found in \cite{solomyak:SS, AP}.

\begin{proposition}
\label{prop:subsMPH}
The map $\Phi_S:\Omega_{S} \to \Omega_{S}$ is a homeomorphism, and the action of $\mathbb{R}^{d}$ on $\Omega_{S}$ is uniquely ergodic.
\end{proposition}

\section{Cohomology}
\label{sec:cohom}
In this section we review several cohomology theories of pattern spaces.
\subsection{Pattern equivariant functions and Cohomology \cite{KellendonkPEF}}
\label{subsec:PEC}
Let $\Lambda \subset \mathbb{R}^d$ be a Delone set and $\Omega_\Lambda$ its associated pattern space. We say that $f:\mathbb{R}^d\rightarrow\mathbb{C}$ is (strongly) $\Lambda$-equivariant if for some $r > 0$
$$B_{r} \cap \varphi_x(\Lambda) = B_{r} \cap \varphi_y(\Lambda) \Rightarrow f(x) = f(y).$$
\begin{remark}
There is a notion of \emph{weakly} $\Lambda$-equivariant functions defined in \cite{KellendonkPEF}. Since we will only use strongly pattern equivariant functions, we will omit the adverb ``strongly'' and just write ``$\Lambda$-equivariant''.
\end{remark}
We denote by $C^k_{\Lambda}(\mathbb{R}^d,V)$ the space of $\Lambda$-equivariant $C^k$ functions on $\mathbb{R}^d$ with values in the vector space $V$. Smooth $\Lambda$-equivariant forms can be then identified with elements of $C^\infty_{\Lambda}(\mathbb{R}^d,\bigwedge(\mathbb{R}^{d})^*)$. We will denote those spaces by $\Delta_{\Lambda}^*$ and by $\star$ the Hodge-$\star$ operator which induces an isomorphism $\star:\Delta_\Lambda^k\rightarrow\Delta_\Lambda^{d-k}$ for any $k\in\{0,\dots, d\}$. We will denote by $(\star 1)$ the canonical Lebesgue volume form in $\mathbb{R}^d$.

The spaces $\Delta_\Lambda^k$ form a subcomplex of the de Rham complex. As such, we can define the pattern equivariant cohomology
\begin{equation}
\label{eqn:PEcoh}
H^k(\Omega_\Lambda;\mathbb{C}) = \frac{\ker \{d:\Delta_{\Lambda}^k\rightarrow \Delta_{\Lambda}^{k+1}   \}}{\mathrm{Im}\{ d:\Delta_{\Lambda}^{k-1}\rightarrow \Delta_{\Lambda}^{k}\}}.
\end{equation}

\begin{definition}
A function $f: \Omega_\Lambda\rightarrow \mathbb{C}$ is \emph{transversally locally constant} if for every $\Lambda'\in\Omega_\Lambda$ there is a $\epsilon>0$ such that $f$ is constant on $\mathcal{C}_{\Lambda',\epsilon,\{\bar{0}\}}$.
\end{definition}
We denote by $C^\infty_{tlc}(\Omega_\Lambda, \bigwedge \mathbb{C}^d)$ the transversally locally constant functions in $C^\infty(\Omega_\Lambda, \bigwedge \mathbb{C}^d)$ which are smooth along the leaves of the foliation.
\begin{theorem}[\cite{Kellendonk-Putnam:RS}]
\label{thm:tlc}
Every smooth $\Lambda $-equivariant function $g:\mathbb{R}^d\rightarrow \mathbb{C}$ defines a unique smooth, transversally locally constant function $\bar{f}_g:\Omega_\Lambda\rightarrow\mathbb{C}$ and we have that $\varphi^*_t \bar{f}_g(\Lambda) = g(t)$. This induces an algebra isomorphism between transversally locally constant functions which are smooth along leaves, and smooth $\Lambda$-equivariant functions.
\end{theorem}
We shall denote by
\begin{equation}
\label{eqn:AlgIso}
i_\Lambda: C^\infty_{tlc}(\Omega_\Lambda)\rightarrow \Delta^0_{\Lambda}
\end{equation}
the isomorphism given by above theorem.

\subsection{\v{C}ech cohomology}
\label{subsec:Cech}
For clarity, this subsection will use the setting of tilings instead of Delone sets. We will make some mild assumptions on our tilings: that our tiling is built out of polyhedra for which there are only finitely many tile types (up to translation) with tiles meeting full face to full face. This assumption gives the tiling finite local complexity. Note that any such tiling may be associated with a Delone set satisfying finite local complexity; see Remark \ref{rmk:backandforth}.

Let $T$ be a tiling of $\mathbb{R}^{d}$ satisfying the conditions above. We let $\check{H}^{*}(\Omega_T;\mathbb{C})$ denote the \v{C}ech cohomology of $\Omega_T$ with coefficients in $\mathbb{C}$. It is now well known that it is possible to compute $\check{H}^{*}(\Omega_\Lambda,\mathbb{C})$ using the inverse limit structure of $\Omega_{T}$, described below. While there are many presentations of $\Omega_T$ as an inverse limit, we use the method initially outlined by Anderson-Putnam and generalized by G\"{a}hler. More information can be found in \cite{sadun:inverse}, which our presentation follows.

Suppose $T$ is our given tiling, and let $P$ be a patch in $T$. The first corona of $P$ is defined to be all tiles $t$ in $T$ for which $t \cap P \ne \varnothing$. The first corona is thus a new patch $P_{1}$ in $T$, and we may consider the second corona $P_{2}$ of $P$ to be the patch obtained as the first corona of the patch $P_{1}$. In this way, given the $n$-th corona $P_{n}$ of $P$, we define the $n+1$st corona of $P$ to be the patch $P_{n+1}$ defined as the first corona of $P_{n}$.

Let $n \in \mathbb{N}$. We say two tiles $t_1, t_2$ in $T$ are $n$-equivalent if a patch of $T$ consisting of $t_1$ and it's $n$-th corona is equal, up to translation, to a similar patch around $t_2$. Because we assume $T$ has finite local complexity, there are only finitely many $n$-equivalence classes of tiles, and we refer to such a class as an $n$-collared tile. Thus an $n$-collared tile consists of a tile, along with all the information of its $n$-th corona. Let $\{p_{i}^{(n)}\}_{i=1}^{m(n)}$ denote the collection of $n$-collared tiles in $T$. Consider now the disjoint union $P^{(n)} = \bigsqcup_{i=1}^{m(n)}p_{i}^{(n)}$, and consider the relation on $P^{(n)}$ obtained by identifying two boundary faces of $n$-collared tiles $p_{j}^{(n)}$ and $p_{k}^{(n)}$ in $P^{(n)}$ if there is a patch in $T$ for which those boundary faces meet. We let $\Gamma_{n}$ denote the quotient of $P^{(n)}$ under this relation, and refer to it as the $n$-th Anderson-Putnam complex (for $T$). The spaces $\Gamma_{n}$ also come equipped with a branched manifold structure, allowing the notion of smooth functions, and smooth $k$-forms on $\Gamma_{n}$ to be defined (see \cite[\S 5.1]{sadun:book}). We refer to the collection of smooth $k$-forms on $\Gamma_{n}$ as $\Delta^{k}(\Gamma_{n})$. \\
\indent There is a map $f_{n}:\Gamma_{n+1} \to \Gamma_{n}$ given by 'forgetting' the $n+1$-st corona; that is, for $x \in \Gamma_{n+1}$ lying in an $n+1$-st collared tile, we send $x$ to the corresponding $n$-collared tile it lies in. This map extends naturally to branch points in $\Gamma_{n}$, and we get an inverse system $\{\Gamma_{n},f_{n}\}$. \\
\indent The following, due to Anderson and Putnam in the case of substitutions and later G{\"a}hler in more generality, shows that this construction can be used to present the tiling space $\Omega_{T}$ as an inverse limit. See \cite[Ch. 2]{sadun:book}.
\begin{theorem}
\label{APGTheorem}
$\Omega_{T}$ is homeomorphic to $\varprojlim\{\Gamma_{n},f_{n}\}$.
\end{theorem}
Given $n \in \mathbb{N}$, we may use the tiling $T$ of $\mathbb{R}^{d}$ to construct a map $\pi_{n}:\mathbb{R}^{d} \to \Gamma_{n}$: for $x \in \mathbb{R}^{d}$, $\pi_{n}(x) \in \Gamma_{n}$ is obtained by inspecting the $n$-th corona of $x$ in $T$. A key fact, which is not hard to show, is that a function $f$ on $\mathbb{R}^{d}$ is pattern-equivariant if and only if there is some $n$ for which there exists some smooth function $g$ on $\Gamma_{n}$ such that $f = \pi_{n}^{*}(g) = g\circ\pi_{n}$. The following, found in \cite[Theorem 5.4]{sadun:book}, follows immediately from this.
\begin{proposition}
\label{prop:PEformsPullBack}
The collection of pattern-equivariant $k$-forms $\Delta_{T}^{k}$ is naturally identified with $\bigcup_{n \in \mathbb{N}}\pi_{n}^{*}(\Delta^{k}(\Gamma_{n}))$.
\end{proposition}
This can be used to show the following, originally proved by Kellendonk and Putnam \cite{Kellendonk-Putnam:RS}.
\begin{theorem}
\label{IsoCohomology}
For a tiling $T$ of finite local complexity, the pattern-equivariant cohomology $H^{*}(\Omega_T;\mathbb{C})$ is naturally isomorphic to the \v{C}ech cohomology $\check{H}^{*}(\Omega_{T};\mathbb{C})$.
\end{theorem}
\subsection{Cohomology for cut and project sets}
\label{subsec:CAPScohom}
In this section we summarize the results of \cite{FHK:topological, GHK:cohomology} which are relevant to our work. Namely, we review sufficient conditions under which the cohomology spaces defined in (\ref{eqn:PEcoh}) are finite dimensional for cut and project constructions.

Recall from \S \ref{subsec:CAPS} that our CAPS is given by an Euclidean space $E$ of dimension $n$ containing some lattice $\Gamma$ such that $E = \mathbb{R}^d \oplus \mathbb{R}^{n-d}$ with associated projections $\pi_\parallel:E\rightarrow \mathbb{R}^d$ and $\pi_\perp:E\rightarrow \mathbb{R}^{n-d}$. It is assumed that $\mathbb{R}^d$ and $\mathbb{R}^{n-d}$ are in \emph{total irrational position} with respect to $\Gamma$, meaning that $\pi_\parallel$ and $\pi_\perp$ are one to one and with dense image on the lattice. Let $K\subset \mathbb{R}^{n-d}$ be the window which we will take to be a finite union of compact non-degenerate polyhedra in $\mathbb{R}^{n-d}$ with boundary $\partial K$ made up of faces of dimension $n-d-1$. Denote by $\Gamma^\parallel = \pi_\parallel(\Gamma)$ and $\Gamma^\perp = \pi_\perp(\Gamma)$ the rank $n$ subgroups of $\mathbb{R}^d$ and $\mathbb{R}^{n-d}$, respectively.

We denote by
$$\Lambda(\Gamma,K) = \{\pi_\parallel(z):z\in\Gamma\mbox{ and }\pi_\perp(z)\in K\}.$$
For each point $x\in \mathbb{R}^n$ denote
\begin{equation}
\label{eqn:LambdaX}
\Lambda_x(\Gamma,K) = \{\pi_\parallel(z):z\in\Gamma\mbox{ and }\pi_\perp(z+x)\in K\}.
\end{equation}
the point in $\Omega_{\Lambda(K,\Gamma)}$ corresponding to $x$.

The set of \emph{singular points} in $E$ is the set
$$\mathcal{S}(\Gamma,K) = \{ x\in E: \pi_\perp(x)\in \partial K + \Gamma^\perp\}$$
and its complement $N\mathcal{S}$ is the set of \emph{nonsingular points}.

\begin{definition}
We call the CAPS $\Lambda(\Gamma,K)$ \emph{almost canonical} if for each face $f_i$ of the window $K$, the set $f_i + \Gamma^\perp$ contains the affine space spanned by $f_i$.
\end{definition}
This is equivalent \cite{GHK:cohomology} to having a finite family of $(n-d-1)$-dimensional affine subspaces
$$\mathcal{W} = \{W_\alpha\subset \mathbb{R}^{n-d}\}_{\alpha\in I_{n-d-1}}$$
indexed by some finite set $I_{n-d-1}$ such that
\begin{equation}
\label{eqn:Singular}
\mathcal{S}(\Gamma,K) = \mathbb{R}^d + \Gamma^\perp + \bigcup_{\alpha\in I_{n-d-1}} W_\alpha.
\end{equation}
Suppose we have an almost canonical CAPS and we have chosen such a family of subspaces $\mathcal{W}$. An affine subspace $W_\alpha + \gamma \subset \mathbb{R}^{n-d}$ for any $\alpha\in I_{n-d-1}$ and $\gamma\in\Gamma^\perp$ is called a \emph{singular space}. The set of singular spaces is independent of the generating set $\mathcal{W}$ chosen.

Two singular spaces with non-trivial intersection give rise to singular spaces of lower dimensions. As such, $\Gamma$ acts on the set of all singular spaces as follows. If $W$ is a singular space of dimension $r$ and $\gamma\in\Gamma$, then so is $\gamma\cdot W = W + \pi_\perp(\gamma)$. The \emph{stabilizer} $\Gamma^W$ of a singular space is the subgroup $\{\gamma\in\Gamma: \gamma\cdot W = W\}$.

\begin{definition}
\begin{enumerate}
\item For each $0\leq r < n-d$, let $\mathcal{P}_r$ be the set of all singular $r$-spaces. Their orbit space under the translation action by $\Gamma$ is denoted by $I_r = \mathcal{P}_r/\Gamma$.
\item Since $\Gamma$ is abelian, the stabilizer $\Gamma^W$ of a singular $r$-space $W$ depends only on the orbit class $\Theta\in I_r$ of $W$. We denote the stabilizer of an orbit class $\Theta$ by $\Gamma^\Theta$.
\item For $r<k<n-d$ and $W\in \mathcal{P}_k$ of orbit class $\Theta\in I_k$, let $\mathcal{P}^W_r$ be the set $\{U\in\mathcal{P}_r:U\subset W\}$, the set of singular $r$-spaces lying in $W$. Then $\Gamma^\Theta$ acts on $\mathcal{P}^W_r$ and we write $I^\Theta_r = \mathcal{P}^W_r/\Gamma^\Theta$, a set which depends only on the class $\Theta$ of $W$. Therefore $I^\Theta_r\subset I_r$ consists of those orbits of singular $r$-spaces which have a representative which lies in a singular $k$-space of class $\Theta$.
\item We denote the cardinalities of these sets by $L_r = |I_r|$ and $L^\Theta_r = |I_r^\Theta|$.
\end{enumerate}
\end{definition}
Given the collection $\mathcal{W}$ of affine subspaces as defined above, let $\mathcal{N}(\mathcal{W}) = \{u_{W} \mid W \in \mathcal{W}\}$ be a collection of unit vectors for which $u_{W}$ is normal to the subspace $W \in \mathcal{W}$, and such that $\mathcal{N}(\mathcal{W})$ spans $\mathbb{R}^{n-d}$. Following \cite{FHK:topological}, we say $\mathcal{N}(\mathcal{W})$ is decomposable if there exists a partition $\mathcal{N}(\mathcal{W}) = \mathcal{N}_{1} \cup \mathcal{N}_{2}$ such that $\textnormal{span}\mathcal{N}_{1} \cap \textnormal{span}\mathcal{N}_{2} = 0$; if there is no such partition, we call $\mathcal{N}(\mathcal{W})$ indecomposable. Finally, call $\mathcal{W}$ indecomposable if there exists an indecomposable collection $\mathcal{N}(\mathcal{W})$ for $\mathcal{W}$.\\
\indent With all this, we have the following important result \cite{FHK:topological, GHK:cohomology}.
\begin{theorem}
\label{thm:L0andcohomology}
\begin{enumerate}
\item $L_0$ is finite if and only if $H^*(\Omega_\Lambda;\mathbb{R})$ is finite dimensional.
\item If $L_0$ is finite then all the $L_r$ and $L_r^\Theta$ are finite as well, and $\nu = n/(n-d)$ is an integer. Moreover, if $\mathcal{W}$ is indecomposable, then $\mathrm{rank}\,\Gamma^U = \nu\cdot\mathrm{dim}(U)$ for any singular space if and only if $L_0$ is finite.
\end{enumerate}
\end{theorem}
We now go over concrete geometrical conditions which guarantee that the cohomological spaces are finitely generated.
\begin{definition}
A \emph{rational subspace} of $E$ is a subspace spanned by vectors from $\mathbb{Q}\Gamma$. A \emph{rational affine subspace} is a translate of a rational subspace.
\end{definition}

\begin{definition}
\label{def:rational}
A \emph{rational projection method pattern} is any Delone set arising from an almost canonical CAPS satisfying the following rationality conditions:
\begin{enumerate}
\item The number $\nu = n/(n-d)$ is an integer.
\item There is a finite set $\mathcal{D}$ of rational affine subspaces of $E$ in one to one correspondence under $\pi_\perp$ with the set $\mathcal{W}$, i.e., each $W\in\mathcal{W}$ is of the form $W = \pi_\perp(D)$ for some unique $D\in\mathcal{D}$.
\item The elements of $\mathcal{D}$ are $\nu(n-d-1)$-dimensional, and any intersection of finitely many members of $\mathcal{D}$ or their translates is either empty or a rational affine subspace $R$ of dimension $\nu\cdot\mathrm{dim}\,\pi_\perp(R)$.
\end{enumerate}
\end{definition}
For any affine subspace $R$ in $E$, denote by $\Gamma^R$ the stabilizer subgroup of $\Gamma$ under its translation action on $E$.

\begin{theorem}[\cite{FHK:topological, GHK:cohomology}]
\label{thm:finiteCohom}
Suppose $\Lambda(\Gamma,K)$ is a rational projection method pattern and $\mathbb{T}^n = \mathbb{R}^n/\mathbb{Z}^n$. Then
\begin{enumerate}
\item There are isomomorphisms $H^s(\mathbb{T}^n;\mathbb{C}) \cong H^s(\Omega_\Lambda;\mathbb{C})$ for $s< \nu - 1$.
\item For any commutative ring $S$, the cohomology $H^*(\Omega_\Lambda;S)$ is finitely generated over $S$.
\end{enumerate}
\end{theorem}

\section{Renormalization}
\label{sec:renorm}
Recall that a Delone set $\Lambda$ is RFT if the cohomology of its pattern space is finite dimensional and there exists a homeomorphism $\Phi_A$ satisfying the conjugacy (\ref{eqn:mph}).

For a RFT Delone set, the induced action of the homeomorphism $\Phi_A$ on the cohomology is defined through the isomorphism (\ref{eqn:AlgIso}). More precisely, if $h\in \Delta^0_{\Lambda}$ is a $\Lambda$ equivariant function, then $i_\Lambda(\Phi_A^*i^{-1}_\Lambda(h)): \Delta^0_{\Lambda} \rightarrow \Delta^0_{\Lambda}$ is the induced map on pattern equivariant functions. This map can be then extended to a map on $\Lambda$-equivariant forms.
\begin{lemma}
\label{lem:indMap}
The map $i_\Lambda\circ \Phi_A^*\circ i^{-1}_\Lambda: \Delta^0_{\Lambda} \rightarrow \Delta^0_{\Lambda}$ is given by $A^*$, i.e., the following conjugacy diagram holds:
$$
\xymatrix{
C^\infty_{tlc}(\Omega_{\Lambda}) \ar[d]_{i_\Lambda} \ar[r]^{\Phi_A^*} &C^\infty_{tlc}(\Omega_{\Lambda})\ar[d]^{i_\Lambda}\\
\Delta^0_{\Lambda} \ar[r]^{A^*} &\Delta^0_{\Lambda}}$$
\end{lemma}
\begin{proof}
Let $f \in C^\infty_{tlc}(\Omega_{\Lambda})$. It is enough to verify that $i_\Lambda \circ \Phi_{A}^{*}(f) = A^{*} \circ i_\Lambda(f)$. For $t \in \mathbb{R}^{d}$ we have $(i_\Lambda \circ \Phi_{A}^{*}(f))(t) = \Phi_{A}^{*}(f)(\varphi_{t}(\Lambda)) = f(\Phi_{A} \circ \varphi_{t}(\Lambda)) = f(\varphi_{At} \circ \Phi_{A}(\Lambda)) = f(\varphi_{At}(\Lambda))$. On the other hand, we have $(A^{*} \circ i_\Lambda(f))(t) = i_\Lambda(f)(At) = f(\varphi_{At}(\Lambda))$.
\end{proof}
\begin{lemma}
\label{lem:closedForms}
The map $\Phi_A$ preserves exact forms: for $\omega\in \Delta^{d-1}_\Lambda$ there exists $\omega'\in\Delta^{d-1}_\Lambda$ such that $\Phi_A^* \bar{f}_{\star d\omega} = \bar{f}_{\star d \omega '}$.
\end{lemma}
\begin{proof}
In view of Lemma \ref{lem:indMap}, we need to show that given $\omega\in\Delta^{d-1}_\Lambda$, there exists $\omega'\in\Delta^{d-1}_\Lambda$ such that $A^* (\star d\omega) = \star d\omega'$.

If we define $\omega = \sum_{i=1}^d g_i\, \star dx_i$, then it trivially follows that $\star d\omega = \sum_{i=1}^d (-1)^{i+1}X_i g_i$, where $X_i = \frac{d}{dx_i}$. Let $u_i = A^*g_i$ and note that $\omega = \sum_{i=1}^d (A^{-1})^* u_i\, \star dx_i$. We then have
\begin{equation}
\begin{split}
\star d\omega &= \sum_{i=1}^d (-1)^{i+1} X_i [(A^{-1})^* u_i] = \sum_{i=1}^d (-1)^{i+1} (A^{-1})^* [\nabla u_i\cdot A^{-1}_{*,i}] \\
&= (A^{-1})^* \sum_{i=1}^d (-1)^{i+1} \sum_{j=1}^d X_j u_i\cdot A^{-1}_{j,i} \\
&= (A^{-1})^* \sum_{i=1}^d X_i \sum_{j=1}^d (-1)^{j+1} u_j\cdot A^{-1}_{i,j} = (A^{-1})^*\, \star d\omega'.
\end{split}
\end{equation}
It follows that $A^* (\star d\omega) = \star d\omega'$.
\end{proof}

\begin{lemma}
\label{PERescale}
Given $r > 0$, there exists $R(r) = R(r,A) > 0$ such that if $B_{R(r)} \cap \varphi_{x}(\Lambda) = B_{R(r)} \cap \varphi_{y}(\Lambda)$, then $B_{r} \cap \varphi_{Ax}(\Lambda) = B_{r} \cap \varphi_{Ay}(\Lambda)$.
\end{lemma}
\begin{proof}
Fix $r > 0$, and suppose there were no such $R(r)$. Then there exists a sequence $R_{n} \to \infty$ and pairs $x_{n}, y_{n} \in \mathbb{R}^{d}$ such that, for all $n$, $B_{R_{n}} \cap \varphi_{x_{n}}(\Lambda) = B_{R_{n}} \cap \varphi_{y_{n}}(\Lambda)$, but $B_{r} \cap \varphi_{Ax_{n}}(\Lambda) \ne B_{r} \cap \varphi_{Ay_{n}}(\Lambda)$. By compactness we may choose subsequences and assume $\varphi_{x_{n}}(\Lambda)$ converges to some pattern $\Lambda_{1}$ in $\Omega_\Lambda$, and $\varphi_{y_{n}}(\Lambda)$ converges to some pattern $\Lambda_{2}$. Since $B_{R_{n}} \cap \varphi_{x_{n}}(\Lambda) = B_{R_{n}} \cap \varphi_{y_{n}}(\Lambda)$ we have $\Lambda_{1} = \Lambda_{2}$, and hence $\Phi_{A}(\Lambda_{1}) = \Phi_{A}(\Lambda_{2})$. But we also have $\Phi_{A}(\Lambda_{1}) = \Phi_{A}(\lim \Lambda - x_{n}) = \lim \Phi_{A}(\Lambda - x_{n}) = \lim \Lambda - Ax_{n}$, and similarly $\Phi_{A}(\Lambda_{2}) = \lim \Lambda - Ay_{n}$. However, by assumption we have $B_{r} \cap \varphi_{Ax_{n}}(\Lambda) \ne B_{r} \cap \varphi_{Ay_{n}}(\Lambda)$, giving a contradiction.
\end{proof}
By Lemma \ref{lem:closedForms}, the map $A^*$ preserves exact forms, and therefore induces a non-trivial map on the pattern equivariant cohomology $H^*_{\Lambda}$.
\begin{lemma}
\label{lem:agreement}
The map $A^*$ on $\Lambda$-equivariant cohomology agrees with the map $\Phi_A^*$ as induced maps on \v{C}ech cohomology.
\end{lemma}
We will prove Lemma \ref{lem:agreement} in the language of tilings, for presentation reasons; there is no loss in generality, and one can easily adapt the proof to the language of Delone sets. The proof of Lemma \ref{lem:agreement} requires working with the isomorphism in Theorem \ref{IsoCohomology}, which we briefly recall now. Given a closed pattern-equivariant form $\omega \in \Delta_{\Lambda}^{k}$, there exists an $n$ and a closed form $\omega_{n} \in \Delta^{k}(\Gamma_{n})$ such that $\omega = \pi_{n}^{*}(\omega_{n})$ (Proposition \ref{prop:PEformsPullBack}). The class $[\omega_{n}]_{dR}$ in the de Rham cohomology $H^{k}_{dR}(\Gamma_{n};\mathbb{C})$ gives a corresponding class $[\omega_{n}]_{\check{C}}$ in $\check{H}^{k}(\Gamma_{n};\mathbb{C})$, coming from the fact that the de Rham cohomology $H^{k}_{dR}(\Gamma_{n};\mathbb{C})$ is naturally isomorphic to the \v{C}ech cohomology $\check{H}^{k}(\Gamma_{n},\mathbb{C})$ (see \cite{sadun:PECints}). The class $[\omega_{n}]_{\check{C}}$ gives a corresponding class $[\omega_{n}]$ in $\varinjlim\{\check{H}^{k}(\Gamma_{n};\mathbb{C})\}$, and since $\check{H}^{k}(\Omega_{\Lambda};\mathbb{C}) = \varinjlim\{\check{H}^{k}(\Gamma_{n};\mathbb{C})\}$, this gives a class $[\omega_{n}] \in \check{H}^{k}(\Omega_{\Lambda};\mathbb{C})$. This construction gives a map $h:H^{k}(\Omega_\Lambda;\mathbb{C}) \to \check{H}^{k}(\Omega_{\Lambda},\mathbb{C})$ which is the isomorphism described in Theorem \ref{IsoCohomology}.
\begin{proof}[Proof of Lemma \ref{lem:agreement}]
Let $\{\Gamma_{n},f_{n}\}$ be the inverse system described in Theorem \ref{APGTheorem} so that $\Omega_{\Lambda} \cong \varprojlim\{\Gamma_{n},f_{n}\}$. Let $q_{i}:\Omega_{\Lambda} \to \Gamma_{i}$ be the corresponding projections coming from the inverse system. Consider a class $[\omega] \in H^{k}_{\Lambda}(\mathbb{C})$, for which $\omega$ is pattern-equivariant of radius $R$. By Lemma \ref{PERescale}, $A^{*}(\omega)$ is pattern-equivariant of radius $r$ for some $r$. \\
\indent For $n$ large enough, by Lemma \ref{PERescale} there exists $m$ such that the map $A:\mathbb{R}^{d} \to \mathbb{R}^{d}$ given by $A:x \mapsto Ax$ descends to a map $\phi_{A}:\Gamma_{n} \to \Gamma_{m}$. Consider a point $(x_{0},x_{1},\ldots) \in \varprojlim\{\Gamma_{n},f_{n}\}$. The point $x_{n}$ lies in some $n$-collared tile in $\Gamma_{n}$, and we may think of $x_{n}$ as giving instructions for laying an $n$-th collared tile around the origin; call the patch in $\Lambda$ associated to this $n$-collared tile $P_{n}$. Since this patch must occur somewhere in $\Lambda$, there exists $z_{n} \in \mathbb{R}^{d}$ such that $\Lambda - z_{n}$ and $P_{n}$ agree exactly around the origin. Again by Lemma \ref{PERescale}, this implies $\Phi_{A}(\Lambda - z_{n}) = \Lambda - Az_{n}$ and $AP_{n}$ must agree as $m$-collared tiles around the origin. This then implies $q_{m}(\Phi_{A}(x_{0},x_{1},\ldots))$ must agree with $\phi_{A}(x_{n})$ in $\Gamma_{m}$. Note that $n$ and $m$ were independent of $(x_{0},x_{1},\ldots)$, and we get a diagram
$$
\xymatrix{
\Omega_{\Lambda} \ar[r]^{\Phi_{A}} \ar[d]_{q_{n}} & \Omega_{\Lambda} \ar[d]^{q_{m}}\\
\Gamma_{n} \ar[r]^{\phi_{A}} & \Gamma_{m}}$$
Furthermore, by construction, $\phi_{A}$ satisfies $\phi_{A} \circ q_{n} = q_{m} \circ \Phi_{A}$. \\
\indent Now the classes $[\omega], [A^{*}(\omega)]$ correspond to de Rham classes $[\omega]_{dR} \in \Delta^{k}(\Gamma_{n}), [A^{*}(\omega)]_{dR} \in \Delta^{k}(\Gamma_{m})$, and by construction they satisfy $\phi_{A}^{*}([\omega]_{dR}) = [A^{*}(\omega)]_{dR}$. By Theorem 3 in \cite{sadun:PECints}, there is a \emph{natural} isomorphism between the real \v{C}ech cohomology and de Rham cohomology on the branched manifolds $\Gamma_{n}, \Gamma_{m}$; thus, letting $[\omega]_{\check{C}}, [A^{*}(\omega)]_{\check{C}}$ denote the corresponding \v{C}ech classes under this isomorphism, naturality of this isomorphism implies we also get $\Phi_{A}^{*}([\omega]_{\check{C}}) = [A^{*}(\omega)]_{\check{C}}$, as desired. Given the description of the isomorphism $h:H^{k}(\Omega_\Lambda;\mathbb{C}) \to \check{H}^{k}(\Omega_{\Lambda},\mathbb{C})$ above, this completes the proof.
\end{proof}

Denote by $g_{i,j,k} = \star \eta_{i,j,k}$ the collection of pattern equivariant functions associated to $\eta_{i,j,k}$ according to the basis chosen in (\ref{eqn:basisAct}). As such, for any $\eta \in \Delta^d_\Lambda$, there exist real numbers $\{ \alpha_{i,j,k}(\eta) \}$, for $i\in \{1,\dots, r\}$, $j\in \{1,\dots, \kappa(i)\}$ and $k\in \{1,\dots, s(i,j)\}$, and $\omega \in \Delta^{d-1}_\Lambda$ such that
\begin{equation}
\label{eqn:decomp}
\star \eta = \sum_{i=1}^r \sum_{j=1}^{\kappa(i)} \sum_{k=1}^{s(i,j)}\alpha_{i,j,k}(\eta) g_{i,j,k}  + \star d\omega.
\end{equation}
Recall that $\nu_1\geq\cdots\geq |\nu_r|$ are the eigenvalues of the induced action $\Phi_A^*:H^d(\Omega_\Lambda;\mathbb{C})\rightarrow H^d(\Omega_\Lambda;\mathbb{C})$. In view of Lemmas \ref{lem:closedForms} and \ref{lem:agreement}, there exist forms $\zeta_{i,j,k}\in\Delta_\Lambda^{d-1}$ such that
\begin{equation}
\label{eqn:ACtionExact}
A^*\eta_{i,j,k} = \nu_i \eta_{i,j,k} + \eta_{i,j-1,k} + d\zeta_{i,j,k}.
\end{equation}
from which it follows that
\begin{equation}
\label{eqn:fullAct}
(A^*)^n\eta_{i,j,k} = \sum_{q=0}^{\mathrm{min}\{n,j-1\}}  \binom{n}{q}  \nu_i^{n-q} \eta_{i,j-q,k} + \sum_{\ell=0}^{n-q-1} \binom{q+\ell}{q}  \nu_i^\ell (A^*)^{n-\ell-q-1}d\zeta_{i,j-q,k}.
\end{equation}

We note that we always have
\begin{equation}
\label{eqn:firstEig}
\nu_1 = \det A
\end{equation}
since $A^* (\star 1) = \det A\, (\star 1)$. By the Hodge-$\star$ isomorphism, this corresponds to the subspace consisting of constant functions in $L^2(\Omega_\Lambda,\mu)$. Let $\mathcal{D}_+ = \det A$, $\mathcal{D}_- = \mathcal{D}_+^{-1}$.

\section{Currents and ergodic integrals}
\label{sec:erg}
Let $\Lambda\subset \mathbb{R}^d$ be a Delone set such that its associated pattern space $\Omega_\Lambda$ is compact and the $\mathbb{R}^d$ action on it is uniquely ergodic. This is satisfied if we assume that the Delone sets which we work with have finite local complexity and uniform cluster frequency.

\subsection{Averaging sets and (co)boundaries}
\label{subsec:avSets}
Let $\mathcal{H}^m$ denote the $m$-dimensional Hausdorff measure.
\begin{definition}
A set $E\subset \mathbb{R}^d$ is called \emph{$m$-rectifiable} if there exist Lipschitz maps $f_i: \mathbb{R}^m\rightarrow\mathbb{R}^d$, $i = 1,2,\dots$ such that
$$\mathcal{H}^m\left(  E\backslash \bigcup_{i\geq 0} f_i(\mathbb{R}^m)   \right) = 0.$$
\end{definition}
\begin{definition}
A \emph{Lipschitz domain} $A\subset\mathbb{R}^d$ is an open, bounded subset of $\mathbb{R}^d$ for which there exist finitely many Lipschitz maps $f_i:\mathbb{R}^{d-1}\rightarrow \mathbb{R}^d$, $i = 1,\dots, L$ such that
$$\mathcal{H}^{d-1}\left(  \partial A \backslash \bigcup_{i=1}^L f_i(\mathbb{R}^{d-1})   \right) = 0.$$
\end{definition}
Clearly, Lipschitz domains have $d-1$-rectifiable boundaries.
\begin{definition}
An open subset $A\subset\mathbb{R}^d$ is a \emph{good Lipschitz domain} if it is a Lipschitz domain and $\mathcal{H}^{d-1}(\partial A)<\infty$.
\end{definition}

Let $\Lambda$ be a RFT Delone set and $A\in GL^+(d,\mathbb{R})$ be the associated expanding matrix. There exists $a\in \mathfrak{gl}(d,\mathbb{R})$ which satisfies $\exp (a) =A$ (as pointed out in the introduction, we assume this without loss of generality because we could pass to $A^2$ if it is not true for $A$ without affecting the results). For $T\geq 0$, let $g_T = \exp(aT)$. Given a bounded measurable set $F_0$ with non-empty interior, we define the averaging family $\{B_T\}_{T\geq 1}$ by
$$B_T = F_{\sigma \log T} = g_{\sigma \log T} F_0,$$
where $\sigma  = d/\log \det A$. As such, we have that $\mathrm{Vol}(B_T) = \mathrm{Vol}(F_0) \mathcal{D}_+^{\sigma \log T}  = \mathrm{Vol}(F_0) T^d$.

Suppose $\rho:\mathbb{R}^d\rightarrow\mathbb{R}$ is a $\Lambda$-equivariant function. Moreover assume that $\rho\in \star d\Delta^{d-1}_{\Lambda}$, where $\star$ denotes the Hodge-$\star$ operator. Note that by Theorem \ref{thm:tlc} there exists a function $\bar{f}_{\rho}\in C^0(\Omega_\Lambda)$ such that $\rho(t) = \varphi_t^* \bar{f}_{\rho}(\Lambda_0)$. By Stokes' theorem, for any bounded, measurable subset $A\subset \mathbb{R}^d$ we have that
\begin{equation}
\label{eqn:cbdry}
\int_A \varphi^*_t \bar{f}_\rho (\Lambda_0)\, dt = \int_{A } \rho(t)\, dt  = \int_{ A} \star d\omega\, dt = \int_{\partial A} \omega
\end{equation}
for some $\omega \in \Delta_{\Lambda}^{d-1}$. By Poincar\'{e}'s Lemma, for any continuous function $g:\mathbb{R}^d\rightarrow\mathbb{R}$, one can always find a solution $\omega$ to the equation $g = \star d\omega$. However, only if $g\in \star d\Delta^{d-1}_{\Lambda}$ is the solution $\omega$ going to be pattern equivariant. We call functions in $\star d\Delta^{d-1}_{\Lambda}$ \emph{coboundaries}.
\begin{lemma}
\label{lem:cbdry}
Let $F_0$ be a good Lipschitz domain and define the one-parameter family of sets $B_T$ according to (\ref{eqn:rescalledSets}). For any coboundary $\psi$ there exists a constant $K = K(F_0,\psi)$ such that for all $T>1$,
$$\left|\int_{B_T} \varphi^*_t \bar{f}_\psi (\Lambda_0)  \, dt\right|\leq K \cdot T^{d\left(1-\frac{\log|\lambda_d|}{\log\nu_1}\right)}.$$
\end{lemma}
\begin{proof}
Since $\psi$ is a coboundary, by definition, there exists a $\omega\in \Delta^{d-1}_\Lambda$ such that $\psi = \star d\omega$. Denote $\omega  = \sum_{i=1}^d \omega_i \star dx_{i}$, where $\omega_i\in\Delta_\Lambda^0$.

By (\ref{eqn:cbdry}) we have that
\begin{equation}
\label{eqn:bdryEst1}
\left|\int_{B_T} \varphi^*_t \bar{f}_\psi (\Lambda) \, dt\right| = \left|\int_{\partial B_T} \omega  \right|\leq    d\max_i \{\|\omega_i \|_{\infty}\} \cdot\mathcal{H}^{d-1}(\partial B_T).
\end{equation}
Since $F_0 = B_0$ is a good Lipschitz domain, its boundary is $d-1$-rectifiable and $\mathcal{H}^{d-1}(\partial B_0)<\infty$ \cite[Ch. 15]{mattila:GMT}. Since $A$ is an expansive matrix with eigenvalues $\lambda_1, \dots, \lambda_d > 1$, $\sigma = d/\log\det A$, and $\sum_{i=1}^d\log |\lambda_i| = \log \det A$, we have that
$$\det g_{\sigma \log T} = \mbox{exp}\, \left( \frac{d\log T }{\log\det A}\sum_{i=1}^d \log|\lambda_i| \right) =  T^d$$
for all $T\geq 1$. As such, the expansion induced on the by $g_{\sigma \log T}$ on the $d-1$-dimensional Lebesgue measure $\mathcal{H}^{d-1}$ is at most
$$\mbox{exp}\,\left( \frac{d\log T}{\log\det A}  \sum_{i=1}^{d-1}\log|\lambda_i|\right) = \mbox{exp}\,\left( \frac{d\log T}{\log\det A}  \log\left(\frac{\det A}{|\lambda_d|} \right)\right) = \mbox{exp}\,\left( d\log T \left( 1 - \frac{\log |\lambda_d|}{\log\det A}   \right)\right)$$
and, since $\det A = \nu_1$, it follows that
$$ \mathcal{H}^{d-1}(\partial B_T) \leq \mathcal{H}^{d-1}(\partial B_0)  \mbox{exp}\,\left( \frac{d\log T}{\log\det A}  \log\left(\frac{\det A}{|\lambda_d|} \right)\right)   = \mathrm{Vol}^{d-1}(\partial B_0) T^{d\left( 1-\frac{\log |\lambda_d|}{\log\nu_1}\right)}.$$
This, combined with (\ref{eqn:firstEig}) and (\ref{eqn:bdryEst1}), gives the desired result.
\end{proof}

\begin{remark}
\label{rem:FZ}
The boundary effects appearing above only appear for higher rank actions. This is a phenomenon which also appears in \cite{BufetovSolomyak}. Lemma \ref{lem:cbdry} suggests that if we want to get meaningful information about growth of ergodic integrals, we should look at functions which are not coboundaries, i.e., functions coming from the pattern equivariant forms in $\star\mathrm{coker}\{d:\Delta^{d-1}_{\Lambda}  \rightarrow \Delta^{d}_{\Lambda}\}$. Note that such functions are given, by definition, by representatives of classes in $H_\Lambda^d(\Omega_\Lambda;\mathbb{R})$. In other words, \emph{the growth of ergodic averages is controlled by the cohomology class of the function}. This is the spirit of the Zorich-Forni phenomenon for ergodic averages of translation flows on flat surfaces.
\end{remark}
\begin{definition}
Given an expanding, diagonalizable $d\times d$ matrix $A$, a \emph{time cube} $F_0$ is a subset of $\mathbb{R}^d$ which
\begin{enumerate}
\item is a finite polytope with an even number of faces;
\item for every face $\partial_i^+ F_0 \subset \partial F_0$ there is another face $\partial^-_i F_0$ such that $\partial^-_i F_0 = \partial^+_i F_0 + c_i$ for some $c_i\in\mathbb{R}^d$;
\item the face $\partial_i^+ F_0$ is orthogonal to an eigenvector $e_{\ell(i)}$ of $A$.
\end{enumerate}
\end{definition}
\begin{remark}
  Note that in the self-similar case ($A$ is pure dilation), every vector in $\mathbb{R}^d$ is an eigenvector of $A$, and thus requirement (3) from the definition is automatically satisfied. Thus, in the self-similar case, any polytope satisfying the first two conditions is a time cube.
\end{remark}
Any collection of $d-1$ linearly independent vectors in $\mathbb{R}^d$ defines an action of $\mathbb{R}^{d-1}$ on $\Omega_\Lambda$, which is a subaction of the $\mathbb{R}^d$ action. Any such subaction preserves the measure $\mu$ which is invariant for the full $\mathbb{R}^d$ action on $\Omega_\Lambda$.
\begin{definition}
A time cube is \emph{stellar} if all the $\mathbb{R}^{d-1}$ subactions parallel to the faces of its boundary are uniquely ergodic.
\end{definition}
Note that if the subactions defined by faces of times-cubes are uniquely ergodic, the unique invariant measure coincides with the unique invariant measure $\mu$ of the full $\mathbb{R}^d$ action on $\Omega_\Lambda$.
\begin{lemma}
\label{lem:awesome}
If $B_0$ is a stellar time cube, $B_T$ is defined using (\ref{eqn:rescalledSets}), and $g$ is a coboundary, then
$$\frac{1}{T^{d\left( 1 -\frac{\log |\lambda_d|}{\log \nu_1}\right)}}\int_{B_T} \bar{f}_{g}\circ \varphi_s(\Lambda_0)\, ds \longrightarrow 0.$$
\end{lemma}
\begin{proof}
Since $g$ is a coboundary, $g = \star d\omega$ for some $\omega\in\Delta_\Lambda^{d-1}$. Then
\begin{equation}
\label{eqn:splitBdry}
\begin{split}
\int_{B_T} \bar{f}_{\star d\omega}\circ \varphi_s(\Lambda_0)\, ds &= \int_{B_T} d\omega = \int_{\partial B_T} \omega  = \int_{\bigcup_i \partial_i^\pm B_T} \omega\\
&= \sum_i \int_{\partial^\pm_i B_T} \omega = \sum_i \int_{\partial^+_i B_T}  \omega +  \int_{\partial^-_i B_T}  \omega \\
&= \sum_i \int_{\partial^+_i B_T}  i_{\bar{X}_i}\omega \circ \varphi_t (\Lambda_0)\, ds  +  \int_{\partial^-_i B_T} i_{-\bar{X}_i}\omega \circ \varphi_t  (\Lambda_0)\, ds,
\end{split}
\end{equation}
where $\bar{X}_i$ is the wedge of $d-1$ linearly independent vectors tangent to $\partial_i B_0$ and orthogonal to some eigenvector $e_i$ of $A$. Since the faces $\partial_i F_0$ are orthogonal to an eigenvector of $A$, using (\ref{eqn:rescalledSets}), we have that
\begin{equation*}
\mathrm{Vol}^{d-1}(\partial_i^+ B_T) = \mathrm{Vol}^{d-1}(\partial_i^+ F_0) e^{\frac{d}{\log \det A}\log T \log(\det A /|\lambda_{\ell(i)}|)} = \mathrm{Vol}^{d-1}(\partial_i^+ F_0) T^{d\left(1-\frac{\log |\lambda_i|}{\log\nu_1}\right)},
\end{equation*}
where $\lambda_{\ell(i)}$ is the eigenvalue corresponding to the eigenvector orthogonal to $\partial_i^\pm F_0$ and we have used (\ref{eqn:firstEig}) in the last equality. It follows that the faces with fastest growth are those orthogonal to any eigenvector whose eigenvalue is equal to $\lambda_d$. Let $\partial_d^\pm F_0$ be a pair of such faces. Since a uniquely ergodic action on a compact metric space has uniform convergence of ergodic averages, we have that
\begin{equation*}
\begin{split}
&\frac{1}{T^{d\left(1-\frac{\log |\lambda_d|}{\log\nu_1}\right)}} \left( \int_{\partial^+_d B_T}  i_{\bar{X}_d}\omega \circ \varphi_t (\Lambda_0)\, ds  +  \int_{\partial^-_d B_T} i_{-\bar{X}_d}\omega \circ \varphi_t  (\Lambda_0)\, ds\right)  = \\
& \frac{1}{T^{d\left(1-\frac{\log |\lambda_d|}{\log\nu_1}\right)}} \left( \int_{\partial^+_d B_T}  i_{\bar{X}_d}\omega \circ \varphi_t (\Lambda_0)\, ds  -  \int_{\partial^-_d B_T} i_{\bar{X}_d}\omega \circ \varphi_t  (\Lambda_0)\, ds\right) \rightarrow \mu(\bar{f}_{i_{\bar{X}_d}\omega}) - \mu(\bar{f}_{i_{\bar{X}_d}\omega})=0,
\end{split}
\end{equation*}
which, combined with (\ref{eqn:splitBdry}), yields the result.
\end{proof}

\subsection{Asymptotic Cycles}
We continue our assumption that $\Lambda$ is RFT.
\begin{definition}
A \emph{$\Lambda$-equivariant $k$-current} $\mathfrak{C}$ is an element of the dual space $(\Delta^k_\Lambda)'$. A \emph{$\Lambda$-equivariant distribution} $\mathfrak{D}$ is an element of the dual space of all $\Lambda$-equivariant smooth functions $\Delta^0_\Lambda$.
\end{definition}
\begin{remark}
Every $\Lambda$-equivariant $d$-current $\mathfrak{C}$ defines a $\Lambda$-equivariant distribution $\mathfrak{D}$ by the isomorphism induced by the Hodge-$\star$ operator: $\star\mathfrak{C} = \mathfrak{D}$.
\end{remark}
The only $\Lambda$-equivariant currents which will be relevant in this paper will be $d$-currents and the associated distributions they define by the remark above.

Let $\eta\in\Delta_\Lambda^d$ and $\{ F_i \}_{i>0}$ a van Hove sequence. Define the currents
\begin{equation}
\label{eqn:current1}
\mathfrak{c}_i(\eta) = \frac{1}{\mathrm{Vol}(F_i)}\int_{F_i} \eta = \frac{1}{\mathrm{Vol}(F_i)}\int_{F_i} \varphi^*_t \bar{f}_{\star \eta}\, dt
= \frac{1}{\mathrm{Vol}(F_i)} \left\langle\eta ,F_i\right \rangle .
\end{equation}
By Birkhoff's ergodic theorem,
\begin{equation}
\label{eqn:current2}
\mathfrak{c}_i(\eta) \longrightarrow \mathfrak{C}_\Lambda(\eta) := \int_{\Omega_\Lambda} f_\eta(\Lambda')\, d\mu(\Lambda').
\end{equation}
This defines a $\mathbb{R}^d$-invariant, $\Lambda$-equivariant current $\mathfrak{C}_\Lambda$. Note that by unique ergodicity this does not depend on the van Hove sequence used to obtain the limit current. The following is a consequence of Lemma \ref{lem:cbdry}.
\begin{proposition}
\label{prop:closed}
$\mathfrak{C}_\Lambda$ is a closed current: $\partial\mathfrak{C}_\Lambda(\omega) = \mathfrak{C}_\Lambda(d\omega) = 0$ for any $\omega\in\Delta_\Lambda^{d-1}$.
\end{proposition}
\begin{definition}
The current defined by (\ref{eqn:current2}) is called the \emph{asymptotic cycle}.
\end{definition}
\begin{remark}
The asymptotic cycle $\mathfrak{C}_\Lambda$ represents a non-trivial $\Lambda$-equivariant homology class. Indeed, considering the $d$-form $(\star 1)$, we have that $\mathfrak{C}_\Lambda((\star 1))=1$.
\end{remark}
\begin{remark}
Asymptotic currents defined as in (\ref{eqn:current2}) can be made much more general. For example, for a system which is not uniquely ergodic, there are asymptotic currents defined for any ergodic invariant measure. Moreover, for a higher rank action ($d>1$), there are subactions defined by actions of lower dimensional subspaces of $\mathbb{R}^d$ which lead to asymptotic currents of lower degrees which are invariant for the subactions used to construct it. Such general framework was considered for the asymptotic cycles in \cite{Kellendonk-Putnam:RS}. Here, due to the application we have in mind, we only consider cycles for the full $\mathbb{R}^d$ action on uniquely ergodic pattern dynamical systems.
\end{remark}

\subsection{Growth of ergodic integrals}
\label{sec:warmup}
Let $\Lambda$ be an RFT Delone set with finite local complexity and uniform cluster frequency. For any Delone set $\Lambda_0 := \Lambda \in \Omega_{\Lambda}$ we define a sequence of strongly $\Lambda$-equivariant currents, i.e., linear functionals on the space of all strongly $\Lambda$-equivariant $d$-forms $\Delta_{\Lambda}^d$ by
\begin{equation}
\label{eqn:littleCurrents}
\Upsilon_{\Lambda_0,F_0}^n : \eta \longmapsto  \int_{F_n} \varphi^*_t \bar{f}_{\star\eta} (\Lambda_0)\, dt,
\end{equation}
where $F_n = A^n F_0$, $A$ being the expanding matrix in the conjugacy (\ref{eqn:mph}). Moreover, by (\ref{eqn:current2}), the averaged currents converge to the asymptotic cycle:
$$\frac{1}{\mathrm{Vol}(F_n)}\Upsilon^n_{\Lambda_0,F_0}(\eta) \longrightarrow \mathfrak{C}_{\Lambda}(\eta) = \int_{\Omega_{\Lambda}} \bar{f}_{\star\eta}\, d\mu$$
independent of $\Lambda_0$ by unique ergodicity.

\begin{lemma}
\label{lem:defCurr}
Let $\Lambda_0 := \Lambda \in \Omega_{\Lambda}$ be RFT, $F_0\subset \mathbb{R}^d$ a bounded, measurable subset with non-empty interior, and define the family of sets $\{F_n\}_{n\in\mathbb{N}} = \{A^n F_0\}_{n\in\mathbb{N}}$. For the currents $\Upsilon^n_{\Lambda_0,F_0}$ defined in (\ref{eqn:littleCurrents}) we have that
$$ \Upsilon_{\Lambda_0,F_0}^n = A^n_* \Upsilon_{\Lambda_n,F_0}^{0},$$
where $\Lambda_n = \Phi_A^{-1}(\Lambda)$, for any $n\in\mathbb{N}$.
\end{lemma}
\begin{proof}
If $\eta \in \Delta_\Lambda^d$, then
\begin{equation*}
\begin{split}
\Upsilon_{\Lambda_0,F_0}^n (\eta) &= \int_{F_n}  \bar{f}_{\star \eta}\circ \varphi_\tau(\Lambda_0)\, d\tau  = \int_{F_n}  \bar{f}_{\star \eta}\circ \varphi_\tau \circ \Phi_A(\Lambda_1)\, d\tau  \overset{(i)}{=}
\int_{F_n}  \bar{f}_{\star \eta}\circ   \Phi_A  \circ  \varphi_{A^{-1}\tau}(\Lambda_1)\, d\tau  \\
&= \frac{1}{\mathcal{D}_-}\int_{F_n} \bar{f}_{\star \eta}\circ   \Phi_A  \circ  \varphi_{A^{-1}\tau}(\Lambda_1)\, \mathcal{D}_-\, d\tau  \overset{(ii)}{=}  \mathcal{D}_+\int_{F_{n-1}} \bar{f}_{\star \eta}\circ   \Phi_A  \circ  \varphi_{s}(\Lambda_1)\, ds \\
&\overset{(iii)}{=} \mathcal{D}_+\int_{F_{n-1}} \bar{f}_{A^*(\star \eta)}  \circ  \varphi_{s}(\Lambda_1)\, ds = \mathcal{D}_+\int_{F_{n-1}} \varphi_{s}^* \bar{f}_{\frac{A^* \eta}{\mathcal{D}_+ (\star 1)}}  (\Lambda_1)\, ds\\
&=  \int_{F_{n-1}} \varphi_{s}^* \bar{f}_{\frac{A^* \eta}{ (\star 1)}}  (\Lambda_1)\, ds = \Upsilon^{n-1}_{\Lambda_1,F_0}(A^*\eta) = A_* \Upsilon^{n-1}_{\Lambda_1,F_0}(\eta),
\end{split}
\end{equation*}
where $(i)$ follows from the conjugacy (\ref{eqn:mph}), $(ii)$ follows from the change of variables formula, and $(iii)$ follows from Lemma \ref{lem:indMap}. Iterating this procedure we obtain that $\Upsilon_{\Lambda_0,F_0}^n= A_*^n \Upsilon^{0}_{\Lambda_n,F_0}$.
\end{proof}
\begin{proposition}
\label{cor:indAct}
Let $\eta_{i,j,k}$ represent a class in the basis chosen in (\ref{eqn:basisAct}) with eigenvalue $\nu_i$. For $n\in\mathbb{N}$ we have that
\begin{equation}
\label{eqn:fullAct2}
 \Upsilon_{\Lambda_0,F_0}^n(\eta_{i,j,k}) = \nu_i^n \sum_{q=0}^{\mathrm{min}\{n,j-1\}}  \binom{n}{q}  \nu_i^{-q} \Upsilon_{\Lambda_0,F_0}^0(\star g_{i,j-q,k}) + \sum_{\ell=0}^{n-q-1} \binom{q+\ell}{q}\nu_i^\ell   \Upsilon_{\Lambda_0,F_0}^{n-\ell-q-1}(d\zeta_{i,j-q,k}).
\end{equation}
from which it follows that
\begin{equation}
\label{eqn:bdryError1}
\left|\Upsilon_{\Lambda_0,F_0}^n(\eta_{i,j,k})\right| \leq K_{F_0, \Lambda} n^{j-1}|\nu_i|^n \| g_{i,1,k} \|_\infty
\end{equation}
whenever $|\nu_i|> \nu_1/|\lambda_d|$. If $|\nu_i|= \nu_1/\|lambda_d|$, then
\begin{equation}
\label{eqnberyError2}
\left|\Upsilon_{\Lambda_0,F_0}^n(\eta_{i,j,k})\right| \leq K_{F_0, \Lambda} n^{j}|\nu_i|^n \| g_{i,1,k} \|_\infty,
\end{equation}
where the constant $K_{F_0, \Lambda, A}$ only depends on $F_0$ and the RFT Delone set $\Lambda$.
\end{proposition}
\begin{proof}
Equation (\ref{eqn:fullAct2}) follows by combining (\ref{eqn:fullAct}) with Lemma \ref{lem:defCurr}. We claim that for any $\zeta\in\Delta_\Lambda^{d-1}$ and any $k\in\mathbb{N}$ we have that
\begin{equation}
\label{eqn:bdryGrowth}
\left| \Upsilon_{\Lambda_0,F_0}^k(d\zeta) \right|\leq K_{F_0} \left(\frac{\nu_1}{|\lambda_d|}\right)^k \|\zeta\|_{\infty}
\end{equation}
for some $K_{F_0} > 0$. Indeed, the current $\Upsilon_{\Lambda_0,F_0}^k$ is defined by an integral and we can use Stokes theorem to convert it to an integral of $\zeta$ over the boundary of $A^k F_0$. Since the rate of growth of the boundary is bounded by $\nu_1/|\lambda_d|$ (see the proof of Lemma \ref{lem:cbdry}), (\ref{eqn:bdryGrowth}) follows.

Using (\ref{eqn:fullAct2}) and (\ref{eqn:bdryGrowth}),
\begin{equation}
\label{eqn:FullCurrentEst}
\begin{split}
& \left| \Upsilon_{\Lambda_0,F_0}^n(\eta_{i,j,k})\right| \\
 &\leq |\nu_i|^n  \sum_{q=0}^{\mathrm{min}\{n,j-1\}} \left|  \binom{n}{q}  \nu_i^{-q} \Upsilon_{\Lambda_0,F_0}^0(\star g_{i,j-q,k})\right| + \sum_{\ell=0}^{n-q-1} \left| \binom{q+\ell}{q}\nu_i^{\ell-n}   \Upsilon_{\Lambda_0,F_0}^{n-\ell-q-1}(d\zeta_{i,j-q,k}) \right| \\
&\leq |\nu_i|^n  \sum_{q=0}^{\mathrm{min}\{n,j-1\}} \left|  \binom{n}{q}  \nu_i^{-q} \Upsilon_{\Lambda_0,F_0}^0(\star g_{i,j-q,k})\right| +  \sum_{\ell=0}^{n-q-1} \binom{q+\ell}{q} |\nu_i|^{\ell-n} K_{F_0}\left(\frac{\nu_1}{|\lambda_d|}\right)^{n-\ell-q-1}  \|\zeta_{i,j-q,k}\|_\infty  \\
&\leq |\nu_i|^n\left( C_1 n^{j-1}\max_m \{\|g_{i,m,k}\|_\infty\}+    K'_{F_0} \max_m  \{\|\zeta_{i,m,k}\|_\infty\}   \sum_{q=0}^{j-1}   \sum_{\ell=0}^{n-q-1} \binom{q+\ell}{q}  \left(\frac{\nu_1}{|\nu_i||\lambda_d|}\right)^{n-\ell}  \right).
\end{split}
\end{equation}
Suppose $|\nu_i|=\nu_1/|\lambda_d|$. Then
$$\sum_{q=0}^{j-1}   \sum_{\ell=0}^{n-q-1} \binom{q+\ell}{q}  \left(\frac{\nu_1}{|\nu_i||\lambda_d|}\right)^{n-\ell}= \sum_{q=0}^{j-1}   \sum_{\ell=0}^{n-q-1} \binom{q+\ell}{q} = \sum_{q=0}^{j-1} \binom{n}{n-q-1}$$
which, combining with (\ref{eqn:FullCurrentEst}), gives that
\begin{equation}
\label{eqn:FullCurrentEst1}
\begin{split}
 \left| \Upsilon_{\Lambda_0,F_0}^n(\eta_{i,j,k})\right|
&\leq |\nu_i|^n\left( C_1 n^{j-1}\max_m \{\|g_{i,m,k}\|_\infty\}+    K'_{F_0} \max_m  \{\|\zeta_{i,m,k}\|_\infty\}   \sum_{q=0}^{j-1}\binom{n}{n-q-1}  \right) \\
&\leq |\nu_i|^n\left( C_1 n^{j-1}\max_m \{\|g_{i,m,k}\|_\infty\}+    K''_{F_0} \max_m  \{\|\zeta_{i,m,k}\|_\infty\}  n^j  \right),
\end{split}
\end{equation}
which gives (\ref{eqnberyError2}). We now consider $|\nu_i| > \nu_1/|\lambda_d|$. Let $\vartheta_i = \nu_1/|\nu_i||\lambda_d| < 1$. Given the last line of (\ref{eqn:FullCurrentEst}), we need to estimate the sum
$$\sum_{\ell = 0}^{n-q-1} \binom{q+\ell}{q}  \left(\frac{\nu_1}{|\nu_i||\lambda_d|}\right)^{n-\ell} = \sum_{\ell = 0}^{n-q-1} \binom{q+\ell}{q}  \vartheta_i^{n-\ell},$$
which, after changing indices, we write as
$$\sum_{\ell = q+1}^n \binom{q+n-\ell}{q}\vartheta_i^\ell  = \sum_{\ell = q+1}^n a_i(n,q,\ell).$$
Given that
$$\frac{a_i(n,q,\ell+1)}{a_i(n,q,\ell)} = \frac{n-\ell}{q+n-\ell}\vartheta_i \leq \vartheta_i < 1,$$
for all $\ell = 0,\dots, n-1$, we have that
$$ \sum_{\ell = 0}^{n-q-1} \binom{q+\ell}{q}  \vartheta_i^{n-\ell} = \sum_{\ell = q+1}^n \binom{q+n-\ell}{q}\vartheta_i^\ell \leq \binom{n+1}{q}\sum_{\ell=q+1}^n  \vartheta_i^\ell \leq \frac{\binom{n+1}{q}}{1-\vartheta_i} \leq C_2(i) n^{q},$$
which, combining with (\ref{eqn:FullCurrentEst}), gives that
\begin{equation}
\label{eqn:FullCurrentEst2}
\begin{split}
& \left| \Upsilon_{\Lambda_0,F_0}^n(\eta_{i,j,k})\right| \\
&\leq |\nu_i|^n\left( C_1 n^{j-1}\max_m \{\|g_{i,m,k}\|_\infty\}+    K'_{F_0} \max_m  \{\|\zeta_{i,m,k}\|_\infty\}   \sum_{q=0}^{j-1}   \sum_{\ell=0}^{n-q-1} \binom{q+\ell}{q}  \vartheta_i^{n-\ell}  \right) \\
&\leq |\nu_i|^n\left( C_1 n^{j-1}\max_m \{\|g_{i,m,k}\|_\infty\}+    K'_{F_0} \max_m  \{\|\zeta_{i,m,k}\|_\infty\}   \sum_{q=0}^{j-1}C_2(i) n^q   \right) \\
&\leq |\nu_i|^n\left( C_1 n^{j-1}\max_m \{\|g_{i,m,k}\|_\infty\}+    C_2(F_0,i) \max_m  \{\|\zeta_{i,m,k}\|_\infty\}   n^{j-1}   \right) \leq K_{F_0,\Lambda, A} n^{j-1} |\nu_i|^n,
\end{split}
\end{equation}
which gives (\ref{eqn:bdryError1}).
\end{proof}
Lemma \ref{lem:defCurr} and Proposition \ref{cor:indAct} apply to the currents $\Upsilon^n_{\Lambda_0, F_0}$ only for integer values of $n$. We can make sense of the currents $\Upsilon_{\Lambda_0,F_0}^n$ when $n$ takes non-integer, positive values. It suffices to observe that for any $t\in\mathbb{R}^+$,
\begin{equation}
\label{eqn:fix}
\Upsilon_{\Lambda_0,F_0}^t = \Upsilon_{\Lambda_0,g_{ t\,\mathrm{mod}\, 1} F_0}^{\lfloor t\rfloor } = \Upsilon_{\Lambda_0, F_{t_1}}^{\lfloor t\rfloor },
\end{equation}
where $t_1 = t\,\mathrm{mod}\, 1$ and $g_t = \exp(at)$ with $a\in \mathfrak{gl}(d,\mathbb{R})$ satisfying $\exp (a) =A$. We can thus apply Proposition \ref{cor:indAct} to the current $\Upsilon^n_{\Lambda_0, F_0}$ for any $n\in \mathbb{R}^+$.

Recall that for any function $\bar{f}_{\star\eta} \in C^\infty_{tlc}(\Omega_\Lambda)$ we can decompose it using (\ref{eqn:decomp}) and the currents (\ref{eqn:littleCurrents}) to compute its ergodic integral:
\begin{equation}
\label{eqn:breakCurr}
\int_{F_n}\varphi^*_\tau \bar{f}_{\star\eta}(\Lambda) \, d\tau = \Upsilon_{\Lambda_0,F_0}^n(\eta) =  \sum_{i=1}^r \sum_{j=1}^{\kappa(i)} \sum_{k=1}^{s(i,j)}\alpha_{i,j,k}(\eta) \Upsilon_{\Lambda_0,F_0}^n(\star g_{i,j,k}) + \Upsilon_{\Lambda_0,F_0}^n(d\omega_\eta),
\end{equation}
Recall also that $\nu_i$ are the eigenvalues of the matrix $\mathcal{A}_d$ (see (\ref{eqn:basisAct})). Let $\rho = \dim E^+(\Omega_\Lambda)$.
\begin{proposition}
\label{prop:speeds}
Let $F_0$ be a good Lipschitz domain and let $B_T$ be defined as in (\ref{eqn:rescalledSets}). There exist a constant $C_{F_0,\Lambda}$ and $\rho$ $\mathbb{R}^d$-invariant, $\Lambda$-equivariant closed currents $\{\mathfrak{C}_{i,j,k}\}_{(i,j,k)\in I^+_\Lambda}$ such that for any $f = \bar{f}_{\star \eta }\in C^\infty_{tlc}(\Omega_\Lambda)$ if there is an index $(i,j,k)$ such that $\mathfrak{C}_{i',j',k'} (\eta) = 0$ for all $(i',j',k')<(i,j,k)$ but $\mathfrak{C}_{i,j,k} (\eta) \neq 0$, then for $T>3$,
$$\left| \int_{B_T} \bar{f}_{\star \eta}\circ \varphi_s(\Lambda)\, ds \right| \leq C_{F_0,\Lambda}  L(i,j,T)T^{d\frac{\log|\nu_{i}|}{\log\nu_1}  }\|\bar{f}_{\star \eta}\|_{\infty}.$$
Moreover, if $\mathfrak{C}_{i,j,k}(\eta) = 0$ for all $(i,j,k)\in I^+_\Lambda$, then there exists an $M$ such that
$$\left| \int_{B_T} \bar{f}_{\eta}\circ \varphi_s(\Lambda)\, ds \right| \leq M T^{d\left(1-\frac{\log|\lambda_d|}{\log\nu_1}\right)}\|f\|_{\infty} .$$
for all $T>0$. Finally, if $\mathfrak{C}_{i,j,k}(\eta) = 0$ for all $(i,j,k)\in I^+_\Lambda$ and $F_0$ is an stellar time cube,
$$T^{-d\left(1-\frac{\log|\lambda_d|}{\log\nu_1}\right)} \int_{B_T} \bar{f}_{\eta}\circ \varphi_s(\Lambda)\, ds \longrightarrow 0.$$
\end{proposition}
\begin{remark}
The first current $\mathfrak{C}_{1,1,1}$ is given by the asymptotic cycle $\mathfrak{C}_\Lambda$ in (\ref{eqn:current2}).
\end{remark}
\subsubsection{Proof of Proposition \ref{prop:speeds}}
\label{subsubsec:speeds}
Let $T_1^{(\sigma)} = \sigma \log T\,\mathrm{mod}\,1$. According to (\ref{eqn:decomp}), (\ref{eqn:fix}), (\ref{eqn:breakCurr}), and (\ref{eqn:rescalledSets}) and we have that
\begin{equation}
\label{eqn:expansion0}
\begin{split}
 \int_{B_T}\varphi^*_\tau \bar{f}_{\star\eta}(\Lambda) \, d\tau &= \Upsilon_{\Lambda,F_0}^{\sigma \log T}(\eta) = \sum_{i=1}^r \sum_{j=1}^{\kappa(i)} \sum_{k=1}^{s(i,j)}\alpha_{i,j,k}(\eta) \Upsilon_{\Lambda,F_0}^{\sigma \log T}(\star g_{i,j,k}) + \Upsilon_{\Lambda,F_0}^{\sigma \log T}(d\omega_\eta)  \\
&= \sum_{i=1}^r \sum_{j=1}^{\kappa(i)} \sum_{k=1}^{s(i,j)}\alpha_{i,j,k} (\eta)\Upsilon_{\Lambda,F_{T_1^{(\sigma)}}}^{\lfloor \sigma \log T \rfloor}(\star g_{i,j,k}) + \Upsilon_{\Lambda,F_{T_1^{(\sigma)}}}^{\lfloor \sigma \log T \rfloor}(d\omega_\eta),
\end{split}
\end{equation}
which, by Proposition \ref{cor:indAct} is equal to
\begin{multline}
\label{eqn:expansion}
 \sum_{i=1}^r  \nu_i^{\lfloor \sigma \log T \rfloor} \sum_{j=1}^{\kappa(i)} \sum_{k=1}^{s(i,j)}\alpha_{i,j,k}(\eta)  \sum_{q=j}^{\kappa(i)} \binom{\lfloor \sigma \log T \rfloor}{q-1} \nu_i^{1-q} \Upsilon_{\Lambda,F_{T_1^{(\sigma)}}}^0(\star g_{i,q,k}) + \Upsilon_{\Lambda,F_{T_1^{(\sigma)}}}^{\lfloor \sigma \log T \rfloor}(d\omega_\eta) \\
=  \sum_{i=1}^r   \frac{ T^{d\frac{\log|\nu_i|}{\log \mathcal{D}_+}}  }{\left( \frac{|\nu_i|}{\nu_1^{1/2}}   \right)^{T^{(\sigma)}_1}}    \sum_{j=1}^{\kappa(i)} \sum_{k=1}^{s(i,j)}\alpha_{i,j,k}(\eta)  \sum_{q=j}^{\kappa(i)} \binom{\lfloor \sigma \log T \rfloor}{q-1} \nu_i^{1-q} \Upsilon_{\Lambda,F_{T_1^{(\sigma)}}}^0(\star g_{i,q,k}) + \Upsilon_{\Lambda,F_{T_1^{(\sigma)}}}^{\lfloor \sigma \log T \rfloor}(d\omega_\eta)  \\
= \sum_{i=1}^r     \left[ \frac{\nu_1}{|\nu_i|^{1/2}}   \right]^{T^{(\sigma)}_1}  T^{d \frac{\log|\nu_i|}{\log \mathcal{D}_+} }    \sum_{j=1}^{\kappa(i)} \sum_{k=1}^{s(i,j)}\alpha_{i,j,k}(\eta)  \sum_{q=j}^{\kappa(i)} \binom{\lfloor \sigma \log T \rfloor}{q-1} \nu_i^{1-q} \Upsilon_{\Lambda,F_{T_1^{(\sigma)}}}^0(\star g_{i,q,k}) + \Upsilon_{\Lambda,F_{T_1^{(\sigma)}}}^{\lfloor \sigma \log T \rfloor}(d\omega_\eta).
\end{multline}
Note that by Birkhoff's ergodic theorem, the leading term above is of order $T^d$ which corresponds to the leading eigenvalue of the matrix $\mathcal{A}_d$ which has the space of constant functions as its eigenspace. Consequently we have that
$$d = d\frac{\log\nu_1}{\log \mathcal{D}_+},$$
i.e., that $\log\nu_1 = \log\mathcal{D}_+$ (hence (\ref{eqn:firstEig})). Plugging this back into (\ref{eqn:expansion}):
\begin{multline}
\label{eqn:expansion2}
\int_{B_T}\varphi^*_\tau \bar{f}_{\star\eta}(\Lambda) \, d\tau = \Upsilon_{\Lambda,F_0}^{\sigma \log T}(\eta)  = \\
\sum_{i=1}^r   \left[ \frac{\nu_1}{|\nu_i|^{1/2}}   \right]^{T^{(\sigma)}_1}   T^{d \frac{\log|\nu_i|}{\log \nu_1}}   \sum_{j=1}^{\kappa(i)} \sum_{k=1}^{s(i,j)}\alpha_{i,j,k}(\eta)  \sum_{q=j}^{\kappa(i)} \binom{\lfloor \sigma \log T \rfloor}{q-1} \nu_i^{1-q} \Upsilon_{\Lambda,F_{T_1^{(\sigma)}}}^0(\star g_{i,q,k})  \\
+ \Upsilon_{\Lambda,F_{T_1^{(\sigma)}}}^{\lfloor \sigma \log T \rfloor}(d\omega_\eta).
\end{multline}
For $T>3$, define the $\Lambda$-equivariant currents $\mathfrak{C}^{F_0,T}_{i,j,k}$ as
\begin{equation}
\label{eqn:functl1}
\begin{split}
  \mathfrak{C}_{i,j,k}^{F_0,T} &: \eta\mapsto \mathfrak{C}_{i,j,k}^{F_0,T}(\eta) := \int_{B_T} \varphi_s^* \bar{f}_{\star\eta} (\Lambda)\, ds - \sum_{\substack{(i',j',k')\leq (i,j,k) \\ k'\neq k  }}\alpha_{i',j',k'}(\eta) \Upsilon_{\Lambda,F_{T_1^{(\sigma)}}}^{\lfloor \sigma \log T \rfloor}(\star g_{i',j',k'}) \\
  &=  \int_{B_T} \varphi_s^* \bar{f}_{\star\eta} (\Lambda)\, ds -    \sum_{i'<i} \sum_{j=1}^{\kappa(i')} \sum_{k=1}^{s(i',j)}\alpha_{i',j,k}(\eta) \Upsilon_{\Lambda,F_{T_1^{(\sigma)}}}^{\lfloor \sigma \log T \rfloor}(\star g_{i',j,k})    \\
&\hspace{.7in}   -   \sum_{j'=j+1}^{\kappa(i)} \sum_{k=1}^{s(i,j')}\alpha_{i,j',k}(\eta) \Upsilon_{\Lambda,F_{T_1^{(\sigma)}}}^{\lfloor \sigma \log T \rfloor}(\star g_{i,j',k})    -    \sum_{k'=1,k'\neq k}^{s(i,j)}\alpha_{i,j,k'}(\eta) \Upsilon_{\Lambda,F_{T_1^{(\sigma)}}}^{\lfloor \sigma \log T \rfloor}(\star g_{i,j,k'}).
\end{split}
\end{equation}
In words, the currents $\mathfrak{C}^{F_0,T}_{i,j,k}$ subtract from the ergodic integral $\int_{B_T} \varphi_s^* \bar{f}_{\star\eta} (\Lambda)\, ds$ terms from its expansion in (\ref{eqn:expansion0}) whose norms grow at a faster rate than $L(i,j,T)T^{d\frac{\log |\nu_{i}|}{\log\nu_1}}$. Let $s_i = d\frac{\log |\nu_{i}|}{\log\nu_1}$.
\begin{lemma}
\label{lem:crntBd}
There is a constant $K$ which depends on $F_0$ such that currents $\mathfrak{C}^{F_0,T}_{i,j,k}$, restricted to the rapidly expanding subspace $E^+(\Omega_\Lambda)$, satisfy
$$\left| \mathfrak{C}^{F_0,T}_{i,j,k} (\eta) \right| \leq K L(i,j,T) T^{s_i}  \|\eta\|_{\infty},$$
where $L(i,j,T)$ is defined in (\ref{eqn:logs}).
\end{lemma}
\begin{proof}
Using the definition (\ref{eqn:functl1}) and the expansion (\ref{eqn:expansion0}),
\begin{equation}
\label{eqn:functl2}
\begin{split}
 \mathfrak{C}_{i,j,k}^{F_0,T}(\eta) = & \underbrace{ \alpha_{i,j,k}(\eta) \Upsilon_{\Lambda,F_{T_1^{(\sigma)}}}^{\lfloor \sigma \log T \rfloor}(\star g_{i,j,k})}_{L_{i,j,k}^T(\eta)} + \underbrace{\sum_{j'<j} \sum_{k=1}^{s(i,j')}\alpha_{i,j',k}(\eta) \Upsilon_{\Lambda,F_{T_1^{(\sigma)}}}^{\lfloor \sigma \log T \rfloor}(\star g_{i,j',k})}_{O_{i,j,k}^T(\eta)} \\
&+ \underbrace{\sum_{i'>i} \sum_{j=1}^{\kappa(i')} \sum_{k=1}^{s(i',j)}\alpha_{i',j,k}(\eta) \Upsilon_{\Lambda,F_{T_1^{(\sigma)}}}^{\lfloor \sigma \log T \rfloor}(\star g_{i',j,k})}_{S_{i,j,k}^T(\eta)} +\underbrace{ \Upsilon_{\Lambda,F_{T_1^{(\sigma)}}}^{\lfloor \sigma \log T \rfloor}(d\omega_\eta)}_{E_{i,j,k}^T(\eta)} \\
= & L_{i,j,k}^T(\eta) + O_{i,j,k}^T(\eta) + S_{i,j,k}^T(\eta) + E_{i,j,k}^T(\eta).
\end{split}
\end{equation}
We now bound the terms $L_{i,j,k}^T(\eta), O_{i,j,k}^T(\eta), S_{i,j,k}^T(\eta)$ and $E_{i,j,k}^T(\eta)$. Using Proposition \ref{cor:indAct} and (\ref{eqn:fix}) we have that there exists a constant $K$ depending on $F_0$ and $A$ such that
$$|L_{i,j,k}^T(\eta)| \leq K L(i,j,T) T^{s_i} \| \eta \|_{\infty}.$$
 Similarly it can be worked out that
$$\frac{|O_{i,j,k}^T(\eta)| + |S_{i,j,k}^T(\eta)|}{L(i,j,T) T^{s_i}}\longrightarrow 0$$
if the terms exist. Finally, from Lemma \ref{lem:cbdry} it follows that $|E_{i,j,k}^T(\eta)|  = \mathcal{O}(T^{d\left(1-\frac{\log |\lambda_d|}{\log\nu_1}\right)}) = \mathcal{O}(|\partial B_T|)$. By the definition of the rapidly expanding subspace $E^+(\Omega_\Lambda)$ we have that $T^{s_i} \geq T^{d\left(1-\frac{\log |\lambda_d|}{\log\nu_1}\right)}$, and the Lemma is proved.
\end{proof}
We define $\rho = \dim E^+(\Omega_\Lambda)$ currents on the rapidly expanding subspace $E^+_\Lambda := E^+(\Omega_\Lambda)$ as follows. Pick an index $(i,j,k)\in I^+_\Lambda$. For any $\eta\in \Delta^d_\Lambda$ representing a cohomology class in the rapidly expanding subspace $E^+_\Lambda$, by (\ref{eqn:expansion}), (\ref{eqn:functl2}) and Proposition \ref{cor:indAct} there exists a sequence $T_\ell \rightarrow +\infty$ such that
\begin{equation}
\label{eqn:currentDef2}
\mathfrak{C}^{F_0}_{i,j,k}: \eta \mapsto \mathfrak{C}^{F_0}_{i,j,k}(\eta) := \lim_{\ell\rightarrow \infty}\frac{\mathfrak{C}^{F_0,T_\ell}_{i,j,k}}{L(i,j,T_\ell) T_\ell^{s_i}} \neq 0.
\end{equation}
is well defined. By passing to a subsequence if necessary, since $E^+_\Lambda$ is finite dimensional, the definition of the current can be extended to all of $E^+_\Lambda$. Moreover, by (\ref{eqn:functl1}) and Proposition \ref{cor:indAct} we have that $\mathfrak{C}^{F_0}_{i,j,k}(\eta)$ is a non-zero multiple of $\alpha_{i,j,k}(\eta)$. That these currents are $\mathbb{R}^d$-invariant and closed is immediate. We can define the currents as $\mathfrak{C}_{i,j,k} = \alpha_{i,j,k}$. Then from (\ref{eqn:expansion2}) and Lemma \ref{lem:crntBd} the result follows for the first two bounds; the time cube case follows from Lemma \ref{lem:awesome}.
\begin{proof}[Proof of Theorem \ref{thm:deviations}]
  The currents defined (\ref{eqn:currentDef2}) give, by duality through the Hodge-$\star$ operator, $\Lambda$-equivariant, invariant distributions $\mathfrak{D}^{F_0}_{i,j,k} = \star\mathfrak{C}^{F_0}_{i,j,k}$. By the isomorphism $i_\Lambda:C^\infty_{tlc}(\Omega_\Lambda)\rightarrow \Delta^0_\Lambda$ in Theorem \ref{thm:tlc} the distributions are defined for functions in $C^\infty_{tlc}(\Omega_\Lambda)$ by $\mathcal{D}_{i,j,k}(f) := \alpha_{i,j,k}(\star i_\Lambda(f))$ which, by the preceeding paragraph, is a non-zero multiple of $ \mathfrak{D}^{F_0}_{i,j,k}(i_\Lambda(f)) = \mathfrak{C}^{F_0}_{i,j,k}(\star i_\Lambda(f)) $.

  That $\mathcal{D}_{i,j,k}$ are $\mathbb{R}^d$-invariant follows from the fact that the $\alpha_{i,j,k}$ are $\mathbb{R}^d$ invariant (translations are homotopically trivial homeomorphisms). In the case where one considers the point $\Lambda \in \Omega_{\Lambda}$ , the rest of the proof follows from Proposition \ref{prop:speeds}. For other points, say $\Lambda^{\prime} \in \Omega_{\Lambda}$, given $T>0$, there exists $t_{\Lambda^{\prime},T}$ such that $\int_{B_{T}}f \circ \varphi_{t}(\Lambda) = \int_{B_{T}} f \circ \varphi_{t_{\Lambda^{\prime},T}} \circ \varphi_{t} (\Lambda^{\prime})$, and the rest of the proof follows from Proposition \ref{prop:speeds} applied to the function $f \circ \varphi_{t_{\Lambda^{\prime},T}}$.
\end{proof}
\begin{remark}
  \label{rmk:expansion}
  From the proof of Proposition \ref{prop:speeds}, we can now write the expansion of ergodic integrals written in Remark \ref{rem:expansion1}. Indeed, setting $\mathcal{D}_{i,j,k}(f) := \alpha_{i,j,k}(\star i_\Lambda(f))$, we can define the functions $\Psi_{i,j,k}^{B_0}:\mathbb{R}^+:\longrightarrow \mathbb{R}$ by
  $$\Psi_{i,j,k}^{B_0}(T) = \frac{1}{L(i,j,T)T^{\frac{\log |\nu_i|}{\log\nu_1}}}\int_{B_T}\eta_{i,j,k}$$
  which, by Proposition \ref{cor:indAct}, are bounded. Thus we get the expansion
  \begin{equation}
\label{eqn:intExpansion}
\int_{B_T} f\circ \varphi_s(\Lambda_0)\, ds = \sum_{(i,j,k)\in I^+_\Lambda} \mathcal{D}_{i,j,k}(f)\Psi_{i,j,k}^{B_0}(T)L(i,j,T)T^{d\frac{\log |\nu_{i}|}{\log\nu_1}}  + \mathcal{O}(|\partial B_T|).
\end{equation}
\end{remark}

\section{Diffraction}
\label{sec:diff}
For our pattern dynamical system $(\Omega_\Lambda, \mu, \mathbb{R}^d)$, consider the function defined by $f\in L^2(\Omega_\Lambda,\mu)$ by $t\mapsto (\varphi_t^* f, f)$. This is a positive definite function on $\mathbb{R}^d$ so, by Bochner's theorem, its Fourier transform is a positive measure $\sigma_f$ on $\mathbb{R}^d$, called the \emph{spectral measure} of $f$.

Let $\omega$ be a smooth function with $\omega(0)=1$ compactly supported in an open ball $B_{\varepsilon}(0)$ with $\varepsilon < r_{min}(\Lambda)/2$ and $\int_{\mathbb{R}^d}\omega(t)\, dt = 1$. Define the pattern equivariant function
\begin{equation}
\label{eqn:PEF}
\rho_{\omega,\Lambda} = \omega * \upsilon_\Lambda.
\end{equation}

Let $\gamma_{\omega,\Lambda}$ be the autocorrelation of $\rho_{\omega,\Lambda}$, i.e., $\gamma_{\omega,\Lambda} = (\omega*\tilde{\omega})*\gamma_\Lambda$. The well-known argument of Dworkin \cite{dworkin} relates diffraction measure to spectral measures since
\begin{equation}
\label{eqn:dworkin}
\begin{split}
\gamma_{\omega,\Lambda}(x) &= \lim_{n\rightarrow\infty}\frac{1}{\mathrm{Vol}(F_n)} \int_{F_n} \rho_{\omega,\Lambda} (x+t) \overline{\rho_{\omega,\Lambda} (t)}\, dt \\
&= \lim_{n\rightarrow\infty}\frac{1}{\mathrm{Vol}(F_n)} \int_{F_n} \varphi_x^* \varphi_t^* f_\omega(\Lambda)\cdot \overline{\varphi_t^* f_\omega(\Lambda)}\, dt \\
&= \int_{\Omega_\Lambda} \varphi_x^*f_{\omega}(\Lambda') \overline{f_\omega(\Lambda')}\, d\mu(\Lambda') \\
&= (\varphi_x^*f_\omega,f_\omega),
\end{split}
\end{equation}
and therefore
\begin{equation}
\label{eqn:dworkin2}
\sigma_{f_\omega} = \widehat{\gamma_{\omega,\Lambda}}.
\end{equation}

Denote by $\eta_{\omega,\Lambda}^x$ the $d$-parameter family of $\Lambda$-equivariant $d$-forms
\begin{equation}
\label{eqn:family}
\eta_{\omega,\Lambda}^x = \star  \left(\varphi_x^*\rho_{\omega,\Lambda} \cdot\overline{\rho_{\omega,\Lambda}} \right),
\end{equation}
parametrized by $x\in\mathbb{R}^d$. By (\ref{eqn:current1}), (\ref{eqn:current2}) and (\ref{eqn:dworkin}),
\begin{equation}
\label{eqn:evaluation}
\mathfrak{c}_i(\eta_{\omega,\Lambda}^x) \longrightarrow \mathfrak{C}_\Lambda(\eta_{\omega,\Lambda}^x) = (\varphi_x^* f_{\omega}, f_{\omega}).
\end{equation}
Since by Proposition \ref{prop:closed} $dZ_{\Lambda}^{d-1}\subset \ker\mathfrak{C}_\Lambda$ (where $Z^k_\Lambda$ denotes the set of closed $\Lambda$-equivariant $k$-forms), we can identify $\mathfrak{C}_\Lambda$ with a homology class, i.e., with an element of $\mathrm{Hom}(H^d(\Omega_\Lambda;\mathbb{C}),\mathbb{C})$.
\subsection{Diffraction and asymptotic cycles: Proof of Theorem \ref{thm:homology}}
\label{subsec:homology}
Let $\{\omega_\ell\}_{\ell>0}$ be a sequence of smooth bump functions with the property that $\omega_\ell$ is compactly supported in a ball of radius $r_{min}(\Lambda)/(2\ell)$ and $\int_{\mathbb{R}^d }\omega_\ell\, dx = 1$ for all $\ell$. We have the pointwise convergence of $\Lambda$-equivariant functions on $\mathbb{R}^d$ $\rho_{\omega_\ell,\Lambda} \longrightarrow \upsilon_\Lambda$ and therefore we have the pointwise convergence of functions on $\Omega_\Lambda$, $\bar{f}_{\omega_\ell} \longrightarrow \upsilon_{\mathcal{L}_\Lambda}$, where $\mathcal{L}_\Lambda = \bigcup_{s\in \Lambda}\varphi_s(\Lambda)\subset \mho_\Lambda$.

Let $\eta_\ell^x = \eta^x_{\omega_\ell,\Lambda}$ be the associated families of $\Lambda$-equivariant $d$-forms defined as in (\ref{eqn:family}) corresponding to the sequence $\{\omega_\ell\}$. Note that by the requirement that $\int_{\mathbb{R}^d }\omega_\ell\, dx = 1$ for all $\ell$, the forms $\eta^x_{\omega_\ell,\Lambda}$ represent the same cohomology class. As such, since $\mathfrak{C}_\Lambda$ is closed, $\mathfrak{C}_\Lambda(\eta_\ell^x)$ is independent of $\ell$ and, moreover, by (\ref{eqn:dworkin}), (\ref{eqn:evaluation}) and the dominated convergence theorem, 
\begin{equation}
\label{eqn:dual}
\mathfrak{C}_\Lambda(\eta_\ell^x) \longrightarrow  \mathcal{D}_1(\upsilon_{\mathcal{L}_{\Lambda}}^x) = (\varphi_x^* \upsilon_{\mathcal{L}_{\Lambda}} ,\upsilon_{\mathcal{L}_{\Lambda}}) = \gamma_\Lambda(x).
\end{equation}
In particular, by Proposition \ref{prop:measure}, (\ref{eqn:auto}), and (\ref{eqn:dual}), if $z = x-y$ for some $x,y\in\Lambda$,
$$\mathcal{D}_1(\upsilon_{\mathcal{L}_{\Lambda}}^z) =  (\varphi_x^*\upsilon_{\mathcal{L}_{\Lambda}}, \varphi_y^*\upsilon_{\mathcal{L}_{\Lambda}}) = \mathrm{freq}(x-y,\Lambda) = \tau_\Lambda(P_{x,y}),$$
where $P_{x,y}\subset \Lambda$ is the two point cluster consisting of $x$ and $y$. Therefore, the asymptotic cycle gives rise to a transverse invariant measure through its duality with pattern equivariant cohomology classes.
\subsection{Deviations of diffraction spectra: Proof of Theorem \ref{thm:DiffDevs}}
Let $\{\omega_\ell\}_{\ell>0}$ be a sequence of smooth bump functions and $\eta_\ell^x = \eta^x_{\omega_\ell,\Lambda}$ be the associated families of $\Lambda$-equivariant $d$-forms as in as in \S \ref{subsec:homology}. Let $\mathcal{P}_{\omega_\ell,i,j,k} (x) = \alpha_{i,j,k} (\eta^x_{\omega_\ell,\Lambda})$. 
\begin{lemma}
The functions $\mathcal{P}_{\omega_\ell,i,j,k}$ are supported on $\Lambda-\Lambda$ and independent of $\ell$.
\end{lemma}
\begin{proof}
Since the $\alpha_{i,j,k}$ are closed, and the cohomology class of $\eta^x_{\omega_\ell,i,j,k}$ is independent of $\ell$, the value of $\mathcal{P}_{\omega_\ell,i,j,k}(x)$ does not depend on $\ell$ and therefore, since $\eta^x_{\omega_\ell}\rightarrow \upsilon_\Lambda(\star1)$, we have that $\mathcal{P}_{\omega_\ell,i,j,k}(x) = \alpha_{i,j,k}(\star(\varphi^*_x\upsilon_\Lambda\cdot\upsilon_\Lambda))$. Since $\varphi^*_x\upsilon_\Lambda\cdot\upsilon_\Lambda$ is supported on $\Lambda-\Lambda$, the result follows.
\end{proof}
In view of the above Lemma, we will define $\mathcal{P}_{i,j,k}(x) = \alpha_{i,j,k}(\star(\varphi_x^*\upsilon_\Lambda\cdot\upsilon_\Lambda))$ as the function associated to the index $(i,j,k)$. Since $\mathcal{P}_{i,j,k}$ is bounded and supported in a uniformly discrete set, we denote its Fourier transform as 
\begin{equation}
\label{eqn:diffDistributions}
\widehat{\mathcal{P}_{i,j,k}} = \widehat{\alpha_{i,j,k}(\star(\varphi_x^*\upsilon_\Lambda\cdot\upsilon_\Lambda))} = \sigma_{i,j,k}
\end{equation}
which is a distribution.  By the results of \S \ref{sec:warmup}, (\ref{eqn:intExpansion}) and (\ref{eqn:dworkin}), we have that
$$ \gamma_\Lambda^T(x) = \int_{B_T} \eta_\ell^x  = \sum_{(i,j,k)\in I^+_{\Lambda}} \mathcal{P}_{i,j,k} (x)\Psi_{i,j,k}^{B_0}(T) L(i,j,T)T^{d\frac{\log|\nu_i|}{\log\nu_1}} + \mathcal{O}(|\partial B_T|).$$ 
The result follows.

\section{Substitution systems}
\label{sec:subs}
Recall the setup for substitution systems from \S \ref{subsec:subs} and we pick up here with Proposition \ref{prop:subsMPH}. Let $\mu$ denote the unique invariant probability measure for $\Omega_{S}$. Since $(\Phi_S)_{*}(\mu)$ is also translation invariant, we have $\mu = (\Phi_S)_{*}(\mu)$ by unique-ergodicity, and hence the homeomorphism $\Phi_S:\Omega_{S} \to \Omega_{S}$ is measure preserving. Combining all of this gives the following.
\begin{lemma}
The substitution rule gives a homeomorphism $\Phi_S$ of the type (\ref{eqn:mph}).
\end{lemma}
Recall the construction, outlined in Section 3, of the AP complex $\Gamma_{1}$ associated to $T \in \Omega_{S}$. Anderson and Putnam proved that in the case when $T \in \Omega_{S}$ comes from a substitution, there exists a map $f_{S}:\Gamma_{1} \to \Gamma_{1}$ which can be defined using the substitution $S$, and a continuous surjection $\pi:\Omega_{S} \to \varprojlim\{X,f_{S}\}$ satisfying $\pi \circ \Phi_S = \hat{f_{S}} \circ \pi$, where $\hat{f_{S}}$ denotes the shift map on $\varprojlim\{\Gamma_{1},f_{S}\}$. Furthermore, in the case that the substitution $S$ \emph{forces the border} (see \cite{AP}), the map $\pi$ is a homeomorphism. It is always possible to re-write the substitution $S$ so that it forces the border, and hence we will always assume the substitutions have this property. Using this presentation it is thus possible, as outlined in Section 3, to compute the Cech cohomology of $\Omega_{S}$: since $\Omega_{S} \cong \varprojlim\{\Gamma_{1},f_{S}\}$ and $\check{H}^{k}(\varprojlim\{\Gamma_{1},f_{S}\}) \cong \varinjlim\{H^{k}(\Gamma_{1}),f_{S}^{k,*}\}$, we have $\check{H}^{k}(\Omega_{S}) \cong \varinjlim\{H^{k}(\Gamma_{1}),f_{S}^{k,*}\}$. In particular, since the space $\Gamma_{1}$ is a finite CW-complex, it has finitely generated cohomology, and we get the following. \\
\begin{proposition}
Let $\Omega_{S}$ be the tiling space associated to a primitive aperiodic substitution $S$ having finite local complexity. Then $\check{H}^{*}(\Omega_{S},\mathbb{C})$ is finite rank.
\end{proposition}
\begin{proof}
We have $\check{H}^{*}(\Omega_{S},\mathbb{C}) = \varinjlim\{H^{*}(\Gamma_{1}),f_{S}^{*}\}$. By finite local complexity, $\Gamma_{1}$ is a finite CW-complex, and $H^{*}(\Gamma_{1},\mathbb{C})$ is finite rank. Thus $\check{H}^{*}(\Omega_{S},\mathbb{C})$ is the direct limit of a directed system of complex vector spaces, with each term in the directed system of the same rank, and the result follows.
\end{proof}
\begin{corollary}
Primitive, aperiodic substitution systems with finite local complexity are RFT.
\end{corollary}

\indent Recall the homeomorphism $\Phi_S:\Omega_{S} \to \Omega_{S}$ induces an isomorphism on top degree cohomology $\Phi_S^{*}:\check{H}^{d}(\Omega_{S},\mathbb{C}) \to \check{H}^{d}(\Omega_{S},\mathbb{C})$, which, upon identifying as a linear map between finite rank vector spaces, we may represent as a matrix, which we denote $\mathcal{A}_{d}$.
\begin{proposition}
\label{prop:specCont}
Suppose the substitution $S$ forces the border, and let $M_{S}$ denote the incidence matrix for $S$. The spectrum of the matrix $\mathcal{A}_{d}$ associated to the induced map on cohomology is contained in the spectrum of the substitution matrix $M_{S}$.
\end{proposition}
\begin{proof}
By the relation $\pi \circ \Phi_S = \hat{f_{S}} \circ \pi$, it is enough to consider the spectrum of the map $\hat{f_{S}^{*}}:\check{H}^{*}(\varprojlim\{\Gamma_{1},f_{S}\}) \to \check{H}^{*}(\varprojlim\{\Gamma_{1},f_{S}\})$ induced by the shift map $\hat{f_{S}}$ on $\varprojlim\{\Gamma_{1},f_{S}\}$. Thus we will prove that the spectrum of the map $\hat{f_{S}^{d}}$ on top degree cohomology is contained in the spectrum of $M_{S}$.\\
\indent The homology of $\Gamma_{1}$ may be computed using the cellular complex $C_{\bullet}$ associated to its structure as a CW-complex. The map $f_{S}: \Gamma_{1} \to \Gamma_{1}$ is cellular, and hence induces a chain map on the cellular chain complex $C_{\bullet}$ for $\Gamma_{1}$
$$
\xymatrix{
C_{d} \ar[r]^{\partial_{d}} \ar[d]^{f_{(d),S}} & C_{d-1} \ar[r]^{\partial_{d-1}} \ar[d]^{f_{(d-1),S}} & \cdots & C_{1} \ar[r]^{\partial_{1}} \ar[d]^{f_{(1),S}} & C_{0} \ar[r]^{\partial_{0}} \ar[d]^{f_{(0),S}} & 0 \\
C_{d} \ar[r]^{\partial_{d}}  & C_{d-1} \ar[r]^{\partial_{d-1}} & \cdots & C_{1} \ar[r]^{\partial_{1}} & C_{0} \ar[r]^{\partial_{0}}  & 0 }
$$

The group $C_{d}$ is free abelian on the prototiles, and the map $f_{(d),S}$ is given by the matrix $M_{S}$ (see \cite[pg. 18]{AP}). The induced map on top-degree homology $f_{*,d,S}:H_{d}(C_{\bullet}) \to H_{d}(C_{\bullet})$ is thus given by the the restriction of the action of $M_{S}$ to the invariant submodule $H_{d}(\Gamma_{1},\mathbb{Z}) = H_{d}(C_{\bullet}) = $ker$\partial_{d}$, so the spectrum of $f_{*,d,S}$ is contained in the spectrum of $M_{S}$. Naturality of the universal coefficient theorem for cohomology implies the induced map $f^{*}_{d,S}:H^{d}(\Gamma_{1},\mathbb{C}) \to H^{d}(\Gamma_{1},\mathbb{C})$ is given by the dual of this map $f_{*,d,S}$ tensored with $\mathbb{C}$, and the result follows.
\end{proof}

\section{Renormalization for self-affine  cut and project sets}
\label{subsec:renormCAPS}
Let $\Gamma\subset \mathbb{R}^n = E^\parallel \oplus E^\perp = \mathbb{R}^d\oplus\mathbb{R}^{n-d}$ be a unit covolume lattice in $\mathbb{R}^n$. The \emph{unstable space} $E^u_{\mathcal{A}}$ of $\mathcal{A}\in SL(\Gamma)$ is the span of the generalized eigenspaces of $\mathcal{A}$ corresponding to the eigenvalues $|\lambda_i^+|>1$. We suppose that $E^\parallel$ is a $d$-dimensional, $\mathcal{A}$-invariant subspace of $E^u_{\mathcal{A}}$ and that $E^\perp$ is the direct sum of the complementary eigenspaces so that $E^\parallel\oplus E^\perp = \mathbb{R}^n$. We further suppose that the action of $\mathcal{A}$ restricted to $E^\parallel$ is diagonalizable, i.e., $E^\parallel$ is generated by a basis of eigenvectors:
$$E^\parallel = \langle e_1^+,\dots, e_{d}^+\rangle$$
where $Ae_i^+ = \lambda^+_i e_i^+$. We will denote by $A = \mathcal{A}|_{E^\parallel}$ the action of $\mathcal{A}$ restricted to $E^\parallel$.

Let $K\subset E^\perp$ be a compact set which is the closure of its non-empty interior (a window). Denote by
$$\Lambda(K,\Gamma) = \{\pi^\parallel(x) : \pi^\perp(x)\cap K\neq \varnothing\}$$
 the Delone set associated to this window and lattice and by $\Omega_{\Lambda(K,\Gamma)}$ the pattern space it defines.

We will denote by $\pi_{\Lambda(K,\Gamma)}:\Omega_{\Lambda(K,\Gamma)} \rightarrow \mathbb{T}^n = \mathbb{R}^n/\Gamma$ the projection to the maximal equicontinuous factor (see \cite{BargeKellendonk}). For any $x\in\mathbb{R}^n$ denote by
$$\Lambda_x = \{\pi^\parallel (\gamma + x) \in\mathbb{R}^d: \gamma\in\Gamma,\,\,\,  \pi^\perp(\gamma + x) \in K\}$$
the point in $\Omega_{\Lambda(K,\Gamma)}$ corresponding to this point. Recall from Section \ref{subsec:CAPS} the set of singular points $\mathcal{S}(K,\Gamma)$. The following is found in \cite[Theorem 4.2]{BK}.
\begin{proposition}
\label{prop:onto}
For non-singular $x\in\mathbb{R}^n$, the map  $\pi: \Lambda_x \mapsto x + \Gamma \in \mathbb{T}^n$ extends to a continuous surjection $\tilde{\pi}: \Omega_{\Lambda(K,\Gamma)}\rightarrow \mathbb{T}^{n}$ which is equivariant with the $\mathbb{R}^d$ action and is one-to-one on the set of nonsingular points (modulo $\Gamma$).
\end{proposition}
From here on we write refer to the map as $\pi$ instead of $\tilde{\pi}$, so we have $\pi:\Omega_{\Lambda(K,\Gamma)} \to \mathbb{T}^{d}$.

\indent The pattern metric $\sigma$ on the set $N\mathcal{S}/\Gamma$ is defined by $\sigma(x,y) = d(\Lambda_{x},\Lambda_{y})$, where $d$ denotes the usual pattern metric on Delone sets. It is well known (see for instance \cite{BK}) that for any $x \in N\mathcal{S}$, $\Omega_{\Lambda_{x}}$ is the completion of $N\mathcal{S}/\Gamma$ with respect to the metric $\sigma$.\\
From now on we will assume that our cut and project set is almost canonical. Recall from \S \ref{subsec:CAPScohom} that this allows us, in some cases, to have finite dimensional cohomology.
\begin{lemma}\label{lem:singularpreserve}
Let $\{W_\alpha\}$ be the collection of affine subspaces which generate the singular set, i.e., the collection in (\ref{eqn:Singular}). If for all $\alpha$ there exists a $\beta$ such that $(\mathcal{A}^{-1})^\perp W_\alpha = W_\beta + \Gamma^\perp$, then $\mathcal{A}$ and $\mathcal{A}^{-1}$ both preserve the non-singular set $\mathcal{NS}$.
\end{lemma}
\begin{proof}
We will show that if $x \in \mathcal{S}$, then $\mathcal{A}^{-1}x \in \mathcal{S}$; the argument for $\mathcal{A}$ is analogous, since $\mathcal{A}$ is non-singular. Let $x \in \mathcal{S}$, so $x \in \mathbb{R}^{d} + \Gamma^{\perp} + W_{\alpha}$ for some $\alpha$. It suffices to show that $\pi^{\perp}(\mathcal{A}^{-1}x) \in \Gamma^{\perp} + W_{\beta}$ for some $\beta$. Since we have $\pi^{\perp}(x) \in \Gamma^{\perp} + W_{\alpha}$, write $\pi^{\perp}(x) = \gamma^{\perp} + y$, with $\gamma^{\perp} \in \Gamma^{\perp}$, $y \in W_{\alpha}$. Since $\pi^{\perp}$ is projection on to the $\mathcal{A}$-invariant subspace $E^{\perp}$, we have $\pi^{\perp}\mathcal{A}^{-1}x = \mathcal{A}^{-1}\pi^{\perp}(x) = \mathcal{A}^{-1}(\gamma^{\perp} + y) = \mathcal{A}^{-1}\gamma^{\perp} + \mathcal{A}^{-1}y$. Since $\mathcal{A} \in SL(n,\Gamma)$, $\mathcal{A}^{-1}\gamma^{\perp} = \delta_{1}^{\perp}$ for some $\delta_{1} \in \Gamma$. By assumption, there exists some $W_{\beta}$, $z \in W_{\beta}$ such that $\mathcal{A}^{-1}y = z + \delta_{2}^{\perp}$ for some $\delta_{2} \in \Gamma$. All together we have $\pi^{\perp}\mathcal{A}^{-1}(x) = \mathcal{A}^{-1}\pi^{\perp}(x) = \delta_{1}^{\perp} + z + \delta_{2}^{\perp} \in W_{\beta} + \Gamma^{\perp}$, as desired.
\end{proof}
\begin{remark}
Note that in the case of canonical domains, i.e., in the case when $W_\alpha = \mathrm{span} \pi^\perp\gamma_\alpha$ for some $\gamma_\alpha\in\Gamma$, then the hypothesis of the lemma hold. The Amman-Beenker tiling satisfies this condition, and this example is worked out in \S \ref{subsec:Ammann}. Moreover, rational CAPs satisfy this condition: there exists a finite set $W_i = \pi^\perp \mathcal{D}_i$, each $\mathcal{D}_i$ a rational affine subspace of $E$. Each $E_i$ contains a face $f_i$. Since it's almost canonical, $W_i$ is contained in $f_j + \Gamma^\perp$. But $W_i = \pi^\perp \mathcal{D}_i$ and $\pi^\perp \mathcal{A}^{-1} \mathcal{D}_i\subset \pi^\perp \mathcal{D}_i + \Gamma^\perp,$ since $\mathcal{A}^{-1} \mathcal{D}_i\subset \mathcal{D}_i + \Gamma$. See \cite[\S 4.2]{GHK:cohomology}.
\end{remark}

\begin{lemma}
\label{lem:homeo}
Suppose $\mathcal{A}$ and $\mathcal{A}^{-1}$ preserve the non-singular set $N\mathcal{S}$. The composition of maps $\pi^{-1}_{\Lambda(K,\Gamma)}\circ \mathcal{A} \circ \pi_{\Lambda(K,\Gamma)}:N\mathcal{S}\rightarrow N\mathcal{S}$ extends to a measure-preserving homeomorphism $\Phi_A: \Omega_{\Lambda(K,\Gamma)}\rightarrow \Omega_{\Lambda(K,\Gamma)}$. Moreover, it induces the conjugacy
\begin{equation}
\label{eqn:mph1}
 \Phi_A\circ \varphi_{ t} = \varphi_{A t} \circ \Phi_A
\end{equation}
for any $t\in\mathbb{R}^d$.
\end{lemma}
\begin{proof}
Let $x \in N\mathcal{S}$. The space $\Omega_{\Lambda_{x}}$ can be obtained as the completion of $N\mathcal{S}/\Gamma$ with respect to the pattern metric $\sigma$ on $N\mathcal{S}/\Gamma$. Thus, it is enough to show that the map induced by $\mathcal{A}$ on $N\mathcal{S}/\Gamma$ is uniformly continuous with respect to the metric $\sigma$. We may then extend $\mathcal{A}$ to a homeomorphism $\Phi_{A}$ on the completion.\\
\indent Recall $K$ denotes the window in $E^{\perp}$. Given $R > 0$ and $x \in \mathcal{NS}/\Gamma$, let $\Sigma_{K}(x,R) = \{ \gamma \in \Gamma: \pi^{\perp}(\gamma + x) \in K,  \pi^{\parallel}(\gamma) \in B_{R}(0)\}$. Fix $R > 0$. It suffices to show the following: there exists $\delta > 0$, sufficiently small, such that if $\sigma(x,y) < \delta$, then $\Sigma_{K}(\mathcal{A}x,R) = \Sigma_{K}(\mathcal{A}y,R)$. Indeed, for $\epsilon > 0$, the condition $\sigma(\mathcal{A}x,\mathcal{A}y) < \epsilon$ is equivalent to showing there exists $R > 0$ such that the local patterns for $\Lambda_{\mathcal{A}x}$ and $\Lambda_{\mathcal{A}y}$ agree, up to a small translation, on a ball of radius $R$ around the origin. For $\delta$ small enough, this is equivalent to $\Sigma_{K}(\mathcal{A},R) = \Sigma_{K}(\mathcal{A}y,R)$. \\
\indent Let $\Sigma_{\mathcal{A}^{-1}K}(x,R) = \{\gamma \in \Gamma: \pi^{\perp}(\gamma + x) \in \mathcal{A}^{-1}(K), \pi^{\parallel}(\gamma) \in B_{R}(0)\}$. Choose $\delta > 0$ small enough so that for $x,y \in N\mathcal{S}/\Gamma$, $\sigma(x,y) < \delta$ implies $\Sigma_{\mathcal{A}^{-1}K}(x,R) = \Sigma_{\mathcal{A}^{-1}K}(y,R)$. Then if $\sigma(x,y) < \delta$ and $\gamma \in \Sigma_{K}(\mathcal{A}x,R)$, we have $\pi^{\perp}(\mathcal{A}^{-1}\gamma + x) \in \mathcal{A}^{-1}K$, and hence $\mathcal{A}^{-1}\gamma \in \Sigma_{\mathcal{A}^{-1}K}(x,R)$ (since $|\mathcal{A}z| \ge |z|$ for $z \in E^{\parallel})$. Thus $\mathcal{A}^{-1}\gamma \in \Sigma_{\mathcal{A}^{-1}K}(y,R)$, and hence $\gamma \in \Sigma_{K}(\mathcal{A}y,R)$. The containment $\Sigma(\mathcal{A}y,R) \subset \Sigma(\mathcal{A}x,R)$ is completely analogous. That $\Phi_{A}$ is measure-preserving follows from the fact that $\pi_{\Lambda(K,\Gamma)}$ is almost everywhere one-to-one. The conjugacy (\ref{eqn:mph1}) follows from the $\mathbb{R}^d$-equivariance in Proposition \ref{prop:onto}.
\end{proof}
\subsection{Codimension One Cut and Project}
This subsection is concerned with cut and project systems having dimendion $d$ and codimension one which are associated with a hyperbolic matrix $\mathcal{A} \in SL(\Gamma,d+1)$ as in the previous section. We will assume throughout that we are in the almost-canonical case, so the window $K$ is a finite union of intervals, its faces consist of points, and the collection of singular 0-spaces $P_{0}$ is just a finite union of distinct $\Gamma$-orbits of points. We let $\Lambda = \Lambda(\Gamma,K)$ denote a Delone set in $\mathbb{R}^{d}$ coming from this setup. Throughout, unless said otherwise, cohomology $H^{*}(-)$ denotes \v{C}ech cohomology with real coefficients. Note that in particular, Theorem \ref{thm:finiteCohom} implies the cohomology $H^{*}(\Omega_{\Lambda})$ will be finitely generated. In this setup, the following comes from Lemma \ref{lem:singularpreserve}.
\begin{proposition}
If $\mathcal{A}$ preserves the collection $P_{0}$ of $\Gamma$-orbits of boundary points, then $\Omega_{\Lambda(\Gamma,K)}$ is an RFT Delone set.
\end{proposition}
The following can be found in \cite[\S 4.1]{BKS}.
\begin{theorem}\label{thm:maximalfactorinducedmap}
Assume $\Lambda$ comes from an almost-canonical codimension cut and project scheme. Let $\pi:\Omega_{\Lambda} \to \mathbb{R}^{n}/\Gamma = \mathbb{T}^{n}$ denote the map to the maximal torus factor, as in Proposition \ref{prop:onto}. Then the induced map $\pi^{k,*}:H^{k}(\mathbb{T}^{n}) \to H^{k}(\Omega_{\Lambda})$ is injective for all $0 \le k \le d$.
\end{theorem}
Let $\Phi_{A}$ denote the corresponding homeomorphism of $\Omega_{\Lambda(\Gamma,K)}$ coming from $\mathcal{A}$. In this setup, Theorem \ref{thm:maximalfactorinducedmap} allows us to record the following relationship between $\mathcal{A}$ and the induced map on cohomology.
\begin{proposition}[Spectrum of codimension 1 RFT CAPS]
\label{prop:cod1Spec}
Let $\Lambda$ be a Delone set coming from a codimension one almost canonical cut and project scheme which is RFT, with associated matrix $\mathcal{A} \in SL(\Gamma)$. Let $Sp(\Phi_{A})$ denote the spectrum of the map $\Phi_{A}^{*}:H^{d}(\Omega_{\Lambda}) \to H^{d}(\Omega_{\Lambda})$. Then $Sp(\Phi_{A})$ is the union of the spectrum of the matrix $\mathcal{A}^{-1}$, and a finite collection of roots of unity.
\end{proposition}
\begin{proof}
In degree $d$ we have a diagram of the form
$$
\xymatrix{
H^{d}(\Omega_{\Lambda}) \ar[r]^{\Phi_{A}^{*}} & H^{d}(\Omega_{\Lambda})\\
H^{d}(\mathbb{T}^{n}) \ar[u]^{\pi^{*}} \ar[r]^{\mathcal{A}^{*}} & H^{d}(\mathbb{T}^{n}) \ar[u]^{\pi^{*}}}$$
for which, by Theorem \ref{thm:maximalfactorinducedmap}, the maps $\pi^{*}$ are injective. It follows that $\Phi_{A}^{*}$ preserves the image of $\pi^{*}$. Recall that $n=d+1$. Then $H^{d}(\mathbb{T}^{n}) \cong \mathbb{R}^{{n \choose d}} \cong \mathbb{R}^{n}$, and Poincar\'e Duality implies the map $\mathcal{A}^{*}:H^{d}(\mathbb{T}^{n}) \to H^{d}(\mathbb{T}^{n})$ is isomorphic the map given by multiplication by the matrix $\mathcal{A}^{-1}$. Thus we get that $Sp(\Phi_{A})$ contains the spectrum of $\mathcal{A}^{-1}$. For the remaining part, first consider the subspace $V$ of $H^{d}(\Omega_{\Lambda})$ complementary to $\pi^{*}(\mathbb{T}^{n})$. The space $V$ is generated by $\Gamma$-orbit classes of points in $P_{0}$ (see \cite[\S 5.1]{GHK:cohomology}). Since $\mathcal{A}$ acts as a permutation on these classes, it follows that $\Phi_{A}^{*}$ restricted to the subspace $V \subset H^{d}(\Omega_{\Lambda})$ acts as a permutation matrix, and hence has eigenvalues given by roots of unity.
\end{proof}
The following Proposition implies Theorem \ref{thm:cod1errs}.
\begin{proposition}
\label{prop:cod1errs}
Let $\Lambda(K,\Gamma)$ be a codimension 1 almost canonical RFT cut and project set, $f\in C^\infty_{tlc}(\Omega_{\Lambda(K,\Gamma)})$, $B_0$ a good Lipschitz domain containing the origin and define $B_T$ as in (\ref{eqn:rescalledSets}). Then there exists an $M>0$ such that
$$\left| \int_{B_T} f\circ\varphi_t(\Lambda_0)\, dt - T^d\mu(f)\right|\leq M \log(T)T^{d\left(1- \frac{|\log\lambda_d|}{\log \nu_1} \right)}.$$
\end{proposition}
\begin{proof}
Recall that, since $\Lambda(K,\Gamma)$ is codimension 1, we have that $n =d+1$. Let $\lambda_1,\dots, \lambda_n = \lambda_{d+1}$ be the eigenvalues of $\mathcal{A}$, listed in decreasing order (by norm). Since $\mathcal{A}$ preserves volume and $\Lambda(K,\Gamma)$ is a codimension-1 cut and project set which is self-affine, we have that $|\lambda_1|\geq \cdots \geq |\lambda_d|>1>|\lambda_{d+1}|>0$. Therefore the eigenvalues of $\Phi_{\mathcal{A}}^*:H^d(\Omega_\Lambda)\rightarrow H^d(\Omega_\Lambda)$ which, by Proposition \ref{prop:cod1Spec}, coincide with those of $\mathcal{A}^*:H^d(\mathbb{T}^{n})\rightarrow H^d(\mathbb{T}^{n})$, are given by (cup) products of eigenvalues of $\mathcal{A}^*:H^1(\mathbb{T}^{n})\rightarrow H^1(\mathbb{T}^{n})$, that is, by products of the eigenvalues of $\mathcal{A}$.

We denote the eigenvalues of the induced action on $H^d(\Omega_\Lambda)$ by $\nu_1,\dots, \nu_n$ such that
\begin{equation}
\label{eqn:cod1Evals}
\nu_i = \prod_{j\neq n+1-i}\lambda_j = \lambda^{-1}_{n+1-i}.
\end{equation}
In order for the eigenspace associated to the eigenvalue $\nu_i$ to be in the rapidly expanding subspace, $\nu_i$ needs to satisfy (\ref{eqn:RES}), that is, $|\nu_i|\geq \nu_1/|\lambda_d|$. Using (\ref{eqn:cod1Evals}) this means that we need that $|\lambda_d|\geq \lambda_i$. Since we assumed $|\lambda_i|\geq |\lambda_d|$ for $i<n=d+1$, this can only happen when $|\lambda_d| = |\lambda_i|$, in which case the inequality in (\ref{eqn:RES}) becomes an equality, and rates given by those eigenvalues in Theorem \ref{thm:deviations} are of order $\log(T) T^{d\left(1- \frac{\log\lambda_d}{\log \nu_1}\right) }$.
\end{proof}
\subsection{Higher Codimension Cut and Project}
We assume throughout this subsection that we are in the rational projection method case (see \S \ref{subsec:CAPScohom}). Thus we have a finite collection $\mathcal{D}$ of rational affine subspaces of $E$, and following \cite{GHK:cohomology} we define $\mathbb{A}$ to be the set $\mathcal{R}_{n-d-1} / \Gamma$, where $\mathcal{R}_{n-d-1} = \mathcal{D} + \Gamma$ is the collection of $\Gamma$ orbits of the set $\mathcal{D}$. Corollary 4.10 in \cite{GHK:cohomology} gives the sequence
\begin{equation}
\label{eqn:longcohomologysequence}
\xymatrix{
\cdots \ar[r] & H_{r}(\mathbb{A}) \ar[r]^{j_{*}} &  H_{r}(\mathbb{T}^n) \ar[r]^{m_{*}} & H^{n-r}(\Omega_{\Lambda}) \ar[r] & H_{r-1}(\mathbb{A}) \ar[r] & \cdots \\
}
\end{equation}
where the map $m_{*}$ is identified with the composition $\pi^{*} \circ D$, with $D:H_{r}(\mathbb{T}^n) \to H^{N-r}(\mathbb{T}^n)$ the isomorphism coming from Poincar\'e Duality, and $\pi^{*}:H^{N-r}(\mathbb{T}^n) \to H^{N-r}(\Omega_{\Lambda})$ the map induced by the maximal torus factor map $\pi:\Omega_{\Lambda} \to \mathbb{T}^n$. The space $\mathbb{A}$ is a union of $(n-\nu)$-tori, the tori themselves being $D_{i}/\Gamma^{D_{i}}$, where the $D_{i}$, $i \in I_{n-d-1}$, are the elements of $\mathcal{D}$; furthermore, the intersection of finitely many of these tori is empty, or a subtorus itself (see the comment following \cite[Corollary 4.13]{GHK:cohomology}). Thus, as pointed out in \cite{GHK:cohomology}, computation of $H^{*}(\mathbb{A})$ and the maps $j_{*}$ are possible, in principle, using a Meyer-Vietoris spectral sequence (as in \cite{kalugin}). For our purposes, it is enough to compute these cohomology groups rationally, since we are interested mainly in the spectrum of the map induced by $\Phi_{\mathcal{A}}$ on cohomology. For such rational computations, one needs only to consider ranks, and the tricky extension problems which manifest themselves in the integral computations, treated in \cite{GHK:cohomology}, do not arise.
\begin{lemma}
Assume $\mathcal{A}$ satisfies the assumptions in Lemma 12: that is, assume for all $\alpha$, there exists $\beta$ such that $\mathcal{A}^{-1}W_{\alpha} = W_{\beta} + \Gamma^{\perp}$. Then $\mathcal{A}$ preserves $\mathcal{R}_{n-d-1}$, and hence also preserves $\mathbb{A} = \mathcal{R}_{n-d-1} / \Gamma$.
\end{lemma}
\begin{proof}
For the collection $\mathcal{W} = \{W_{\alpha}\}_{\alpha \in I_{n-d-1}}$, we have a collection of unique rational affine subspaces $\mathcal{D} = \{D_{\alpha}\}_{\alpha \in I_{n-d-1}}$ such that for each $\alpha$, $\pi^{\perp}(D_{\alpha}) = W_{\alpha}$. For $D_{\alpha}+\gamma \in \mathcal{D} + \Gamma$ we have $\pi^{\perp}(\mathcal{A}(D_{\alpha}+\gamma)) = \mathcal{A}\pi^{\perp}(D_{\alpha}) + \pi^{\perp}\mathcal{A}(\gamma) = \mathcal{A}W_{\alpha} +  \pi^{\perp}\mathcal{A}(\gamma) = W_{\beta}+\gamma_{1}^{\perp}$ for some $\gamma_{1}^{\perp} \in \Gamma^{\perp}$. But by the uniqueness of the collection $\mathcal{D}$ we also have $W_{\beta} + \gamma_{1}^{\perp} = \pi^{\perp}(D_{\beta} + \gamma_{1})$, so $\mathcal{A}(D_{i} + \gamma) \in \mathcal{D}+\Gamma$.
\end{proof}
The above lemma, combined with the sequence (\ref{eqn:longcohomologysequence}), implies that the spectrum of the map $\Phi_{A}^{*}:H^{*}(\Omega_{\Lambda};\mathbb{R}) \to H^{*}(\Omega_{\Lambda};\mathbb{R})$ induced by $\Phi_{A}$ may be computed using the homological data coming from the action of $\mathcal{A}$ on $\mathbb{T}^{n}$, and the action of $\mathcal{A}$ on $\mathbb{A}$. To see this, suppose $r$ is fixed, and let $D_{r}:H^{N-r}(\mathbb{T}^n;\mathbb{R}) \to H_{r}(\mathbb{T}^n;\mathbb{R})$ denote the isomorphism coming from Poincar\'e Duality. Letting $\mathcal{A}_{*,r}:H_{r}(\mathbb{T}^n;\mathbb{R}) \to H_{r}(\mathbb{T}^n;\mathbb{R})$ and $\mathcal{A}^{*,N-r}:H^{N-r}(\mathbb{T}^n) \to H^{N-r}(\mathbb{T}^n;\mathbb{R})$ denote the maps induced by $\mathcal{A}$, then, if $det(\mathcal{A}) = 1$, one has the relation (see \cite[\S 3.3]{hatcher}) $\mathcal{A}_{*,r}D_{r}\mathcal{A}^{*,N-r} = D_{r}$ (if $det(\mathcal{A}) = -1$, there is a sign change, since then $\mathcal{A}$ takes the fundamental class $[\mathbb{T}^n] \in H_{N}(\mathbb{T}^n;\mathbb{R})$ to $-[\mathbb{T}^n] \in H_{N}(\mathbb{T}^n;\mathbb{R})$). Furthermore, by construction we have the diagram
$$
\xymatrix{
H^{r}(\Omega_{\Lambda};\mathbb{R}) \ar[r]^{\Phi_{A}^{*,r}} & H^{r}(\Omega_{\Lambda};\mathbb{R})\\
H^{r}(\mathbb{T}^{n};\mathbb{R}) \ar[u]^{\pi^{*,r}} \ar[r]^{\mathcal{A}^{*,r}} & H^{r}(\mathbb{T}^{n};\mathbb{R}) \ar[u]^{\pi^{*,r}}}$$
These two facts, together with the fact that the maps $m^{*}$ in the sequence (\ref{eqn:longcohomologysequence}) are given by the compositions $\pi^{*} \circ D_{r}$ of the inverse of the Poincar\'e Duality isomorphism $D_{r}:H^{N-r}(\mathbb{T}) \to H_{r}(\mathbb{T})$ with the induced maps on cohomology $\pi^{*}:H^{*}(\mathbb{T}) \to H^{*}(\Omega_{\Lambda})$, give the computation of the portion of $H^{*}(\Omega_{\Lambda})$ coming from $Image(m^{*})$ in (\ref{eqn:longcohomologysequence}). But the remaining portion of $H^{*}(\Omega_{\Lambda})$ can be obtained as the kernel of the map $j_{*-1}:H_{*-1}(\mathbb{A}) \to H_{*-1}(\mathbb{T})$ in (\ref{eqn:longcohomologysequence}), and hence one needs only the data of the action coming from the restriction of $\mathcal{A}_{*-1}:H_{*-1}(\mathbb{A}) \to H_{*-1}(\mathbb{A})$ to the kernel of $j_{*-1}$.

\section{Examples and Applications}
\label{sec:ex}
In this section we go over applications of the main results in this paper and compare them with other related results in the literature. We note that although the systems come from aperiodic tilings, it is the vertex set of such tiling which we consider as our Delone sets of interest.
\subsection{The Ammann-Beenker Tiling}
\label{subsec:Ammann}
In this section we study the point set given by the vertex set of the Ammann-Beenker tiling, a well-known aperiodic tiling of $\mathbb{R}^2$ which is given by a substitution rule. Instead of introducing it by its substitution rule, we build up using a rational projection method scheme from toral automorphisms given by matrices in $SL(4,\mathbb{Z})$, which follows the spirit of the presentation of this paper. The interested reader may easily find many references on the tiling as a substitution, and we give a few references, for example, see \cite{HarrissLamb}.

Consider the matrix
$$R_{AB} = \left(\begin{array}{cccc}
0&1&0&0 \\
0&0&1&0 \\
0&0&0&1 \\
-1&0&0&0
\end{array}\right)$$
which induces a rotation of $\pi/4$ on a two-dimensional eigenspace $E_1$ and a rotation of $3\pi/4$ on the other two-dimensional eigenspace $E_2$. The \emph{symmetry group} $GL(A_{AB})$ of $A_{AB}$ is the centralizer of $A_{AB}$ in $GL(n,\mathbb{Z})$. The \emph{special symmetry group} $SL(A_{AB})$ of $A_{AB}$ is the centralizer of $A_{AB}$ in $SL(n,\mathbb{Z})$. All matrices in $GL(A_{AB})$ are of the form
$$M(a,b,c,d) = \left(\begin{array}{cccc}
a&b&c&d \\
-d&a&b&c \\
-c&-d&a&b \\
-b&-c&-d&a
\end{array}\right).$$
Consider the symmetric matrix $\mathcal{A}=M(1,1,0,-1)\in SL(A_{AB})$. It induces a hyperbolic toral automorphism with eigenvalues $1+\sqrt{2}$ with multiplicity 2 and $1- \sqrt{2}$ with multiplicity 2. The subspace generated by the expanding eigenvalues coincides with $E_1$ while the one associated with contracting eigenvalues coincides with $E_2$. Therefore we denote $E^\parallel = E_1$ and $E^\perp=E_2$. Let $v_{i}$ denote the standard basis generating $\mathbb{Z}^{4}$ in $\mathbb{R}^{4}$, so $v_{i}$ has a 1 in the $i$th coordinate, and zeroes elsewhere. We use the octagonal scheme for the acceptance domain as in \cite{GHK:cohomology} and in what follows our notation closely follows their notation. To realize the setup as a rational projection method pattern, first consider the collection $\mathcal{W} = \{W_{i} = span \pi^{\perp}(v_{i})\}_{i=1}^{4}$. This collection generates the set of singular subspaces, with the stabilizers of each $W_{i}$ being two-dimensional, given as follows
$$\Gamma^{W_{1}} = \langle v_{1}, v_{2}-v_{4} \rangle, \hspace{.2in} \Gamma^{W_{2}} = \langle v_{2}, v_{1} + v_{3} \rangle$$
$$\Gamma^{W_{3}} = \langle v_{3}, v_{2} + v_{4} \rangle, \hspace{.2in} \Gamma^{W_{4}} = \langle v_{4}, v_{1} - v_{3} \rangle$$
The family of two-dimensional rational subspaces $\mathcal{D} = \{D_{i} = span_{\mathbb{R}}\Gamma^{W_{i}}\}_{i=1}^{4}$ satisfy $\pi^{\perp}D_{i} = W_{i}$, and we get $\mathbb{A} = \mathcal{D} / \Gamma$ as a collection of four (intersecting) 2-tori. \\

We will use the sequence (\ref{eqn:longcohomologysequence}) to compute the action of $\Phi_{M}$ on $H^{2}(\Omega_\Lambda;\mathbb{R})$, breaking up the calculation in to two parts: the first part coming from $coker(j_{2}:H_{2}(\mathbb{A};\mathbb{R}) \to H_{2}(\mathbb{T}^{4};\mathbb{R}))$, and the second part coming from $ker(j_{1}:H_{1}(\mathbb{A};\mathbb{R}) \to H_{1}(\mathbb{T}^{4};\mathbb{R}))$. Recall the map $m_{2}:\mathbb{T}^{4} \to H^{2}(\Omega_\Lambda;\mathbb{R})$ is obtained by composing the Poincare Duality isomorphism with the pullback map on cohomology coming from the maximal torus factor. \\

\textbf{Part One:} We consider the $v_{i}$'s, $i=1,2,3,4$ as generators for $H^{1}(\mathbb{T}^{4};\mathbb{R})$, and use the basis $\hat{z_{i}}$ coming from cup products (in ascending order) of these for $H^{2}(\mathbb{T}^{4};\mathbb{R})$. We let $z_{i}$ denote the corresponding classes in $H_{2}(\mathbb{T}^{4};\mathbb{R})$ associated to the basis $\hat{z_{i}}$, so in particular $z_{1}$ comes from $\hat{z_{1}} = v_{1} \cup v_{2}$, and $z_{6}$ from $\hat{z_{6}} = v_{3} \cup v_{4}$. It is easy to check using the Mayer-Vietoris spectral sequence that $H_{2}(\mathbb{A};\mathbb{R})$ is rank 4, with generators coming from the standard generator of each of the 4 tori which make up $\mathbb{A}$. Using these bases, the inclusion $j:\mathbb{A} \to \mathbb{T}^{4}$ induces $j_{2}:H_{2}(\mathbb{A};\mathbb{R}) \to H_{2}(\mathbb{T}^{4};\mathbb{R})$ which in matrix form looks like
$$J_{2} = \begin{pmatrix} 1 & -1 & 0 & 0 \\ 0 & 0 & 0 & 0 \\ -1 & 0 & 0 & -1 \\ 0 & 1 & -1 & 0 \\ 0 & 0 & 0 & 0 \\ 0 & 0 & 1 & 1 \end{pmatrix}$$
The matrix $J_{2}$ has rank 3, with the first 3 columns forming a basis for the image. Letting $y_{1}, y_{2}, y_{3}$ denote the first three columns of $J_{2}$, upon defining $y_{4} = z_{1}$, $y_{5} = z_{2}$, $y_{6} = z_{5}$, the collection of $y_{i}$'s extend to form a basis $\mathcal{B} = \{y_{i}\}_{i=1}^{6}$ for $H_{2}(\mathbb{T}^{4};\mathbb{R})$. It follows in particular that the elements $m_{2}(y_{4}), m_{2}(y_{5}), m_{2}(y_{6})$ are not zero, and generate the portion of $H^{2}(\Omega_\Lambda;\mathbb{R})$ coming from the $j_{2}$ part of (\ref{eqn:longcohomologysequence}). One can check directly that the map induced by $M$ on $H_{2}(\mathbb{T}^{4};\mathbb{R})$, written in the basis $\mathcal{B}$, has the form
$$M_{2} = \begin{pmatrix} -1 & 0 & 0 & 0 & 0 & 0\\ 0 & -1 & 0 & 0 & 0 & 0 \\ 0 & 0 & -1 & 0 & 0 & 0 \\ -1 & 2 & 1 & 3 & 1 & 1 \\ -1 & 2 & 1 & 4 & 1 & 2 \\ -1 & 2 & 1 & 4 & 2 & 1 \end{pmatrix}$$
Poincare Duality and $M_{2}$ together induce a map $M^{(2)}:H^{2}(\mathbb{T}^{4};\mathbb{R}) \to H^{2}(\mathbb{T}^{4};\mathbb{R})$ which is conjugate to $M^{-1}_{2}$, since $det(M) = 1$. But by inspection (since $M_{2}$ is block lower triangular) one can check that the map $M_{2}^{-1}$ acts on the subspace spanned by $y_{4},y_{5},y_{6}$ via
$$L = \begin{pmatrix} 3 & 1 & 1 \\ 4 & 1 & 2 \\ 4 & 2 & 1 \end{pmatrix}^{-1} = \begin{pmatrix} 3 & -1 & -1 \\ -4 & 1 & 2 \\ -4 & 2 & 1 \end{pmatrix}$$
Finally, the diagram
$$
\xymatrix{
H^{2}(\Omega_\Lambda;\mathbb{R}) \ar[r]^{\Phi_{M}^{*,2}} & H^{2}(\Omega_\Lambda;\mathbb{R})\\
H^{2}(\mathbb{T}^{4};\mathbb{R}) \ar[u]^{\pi^{*,2}} \ar[r]^{M^{*,2}} & H^{2}(\mathbb{T}^{4};\mathbb{R}) \ar[u]^{\pi^{*,2}}}$$
implies the action of $\Phi_{M}$ on the subspace of $H^{2}(\Omega_\Lambda;\mathbb{R})$ corresponding to $coker(j_{2})$, which is generated by $m_{2}(y_{4}), m_{2}(y_{5}), m_{2}(y_{6})$, is conjugate to the matrix $L$. \\
\indent \textbf{Part Two:} Following the notation for the tori $D_{i}$ which comprise $\mathbb{A}$, we begin by defining the following classes in $H_{1}(\mathbb{A};\mathbb{R})$:
$$\gamma_{1} = [v_{1}], \hspace{.1in} \gamma_{2} = [v_{2} - v_{4}], \hspace{.1in} \gamma_{3} = [v_{2}], \hspace{.1in} \gamma_{4} = [v_{1} + v_{3}]$$
$$\gamma_{5} = [v_{3}], \hspace{.1in} \gamma_{6} = [v_{2} + v_{4}], \hspace{.1in} \gamma_{7} = [v_{4}], \hspace{.1in} \gamma_{8} = [v_{1} - v_{3}]$$
The Mayer-Vietoris spectral sequence associated to $\mathbb{A}$ (see \cite{kalugin}) shows that $rank_{\mathbb{R}}H_{1}(\mathbb{A},\mathbb{R}) = 10$, and one can see directly from the sequence that each of the classes $\gamma_{i}$ is a generator in $H_{1}(\mathbb{A};\mathbb{R})$. For the remaining two generators we use the loops formed by
$$\gamma_{9} = \left\{t \begin{pmatrix} 0 \\ 1 \\ 0 \\ 1 \end{pmatrix} : 0 \le t \le 1/2\right\} +  \left\{t \begin{pmatrix}0 \\ -1 \\ 0 \\ 1 \end{pmatrix} : 1/2 \le t \le 1\right\}$$
$$\gamma_{10} = \left\{s \begin{pmatrix} 1 \\ 0 \\ 1 \\ 0 \end{pmatrix} : 0 \le s \le 1/2 \right\} + \left\{s \begin{pmatrix} -1 \\ 0 \\ 1 \\ 0 \end{pmatrix} : 1/2 \le s \le 1\right\}$$
In this basis for $H_{1}(\mathbb{A};\mathbb{R})$, a basis for the kernel $K$ of $H_{1}(\mathbb{A};\mathbb{R}) \to H_{1}(\mathbb{T}^{4};\mathbb{R})$ is given by
$$a_{1} = -2 \gamma_{1} + \gamma_{4} + \gamma_{8}, \hspace{.1in} a_{2} = \gamma_{2} - \gamma_{3} + \gamma_{7}, \hspace{.1in} a_{3} = \gamma_{2} - 2\gamma_{3}+\gamma_{6}$$
$$a_{4} = \gamma_{1}-\gamma_{4}+\gamma_{5}, \hspace{.1in} a_{5} = \gamma_{7} - \gamma_{9}, \hspace{.1in} a_{6} = \gamma_{5} - \gamma_{10}$$
Now the induced action of $M$ on the kernel $K$ in $H_{1}(\mathbb{A};\mathbb{R})$, in the basis presented above, is given by
$$\begin{pmatrix} 1 & -1 & 0 & 0 & -1 & 1 \\ -2 & 1 & 0 & 0 & 2 & -1 \\ 0 & 0 & 1 & 1 & -1 & 1 \\ 0 & 0 & 2 & 1 & -1 & 2 \\ 0 & 0 & 0 & 0 & -1 & 0 \\ 0 & 0 & 0 & 0 & 0 & -1 \end{pmatrix}$$
(Note we used here the relations $a_{1} + a_{4} - a_{6} = \frac{1}{2}a_{1}$ and $a_{2}-a_{3}-a_{5} = - \frac{1}{2}a_{3}$, which can be checked directly). The spectrum of this matrix consists of: $\{1+\sqrt{2},1+\sqrt{2},1-\sqrt{2},1-\sqrt{2},-1,-1\}$. Note that this is, apart from the -1 and -1, exactly the spectrum of the associated substitution matrix for the Amman-Beenker substitution. \\

Putting together both parts yields the spectrum of $\Phi_{M}$ as the spectrum of $L = \begin{pmatrix} 3 & -1 & -1 \\ -4 & 1 & 2 \\ -4 & 2 & 1 \end{pmatrix},$ which is $(1+\sqrt{2})^2, (1+\sqrt{2})^{-2}, -1$ along with $1+\sqrt{2}$ (with multiplicity 2) and $1-\sqrt{2}$. This shows that the rapidly expanding subspace for this example is three-dimensional, with two eigenvalues not satisfying (\ref{eqn:RES}) strictly.
\subsection{The Penrose Tiling}
Both using the substitution matrix and the matrix representing the induced action on the top level of cohomology (using the Anderson-Putnam complex), with some work, it can be computed that the leading eigenvalues are $\nu_1 = \left(\frac{1+\sqrt{5}}{2}\right)^2$ followed by $\nu_2 = \frac{1+\sqrt{5}}{2} = \nu_1^{\frac{1}{2}}$. Therefore, the rapidly expanding subspace has dimension greater than one, but there are no eigenvalues other than $\nu_1$ which satisfy (\ref{eqn:RES}) strictly. 
\subsection{Large rapidly expanding subspaces}
\label{subsec:largeRES}
Given the previous examples, one wonders whether there are self-affine sets where a (non-leading) eigenvalue for the induced map on cohomology satisfies the inequality (\ref{eqn:RES}) strictly, as opposed to being an equality as in the previous two examples. Here we describe how such examples can be constructed.

Given $d,k \in \mathbb{N}$, it is possible to construct self-affine substitution Delone sets in $\mathbb{R}^d$ with an associated rapidly expanding subspace in cohomology of dimension at least $k$. Indeed, consider a one-dimensional primitive, aperiodic substitution $\sigma$ which is proper (see \cite[\S 6.1]{sadun:book}). Let $\Omega_{\sigma}$ be the corresponding pattern space and $M$ be the incidence matrix matrix which has eigenvalues $\xi_1,\dots, \xi_m$ ordered in decreasing magnitude. Barge and Diamond showed in \cite{BD:cohomology} that $H^{1}(\Omega_{\sigma};\mathbb{R})$ is isomorphic to the direct limit of $M^{t}$ and thus the action of $\sigma$ on $H^{1}(\Omega_{\sigma};\mathbb{R})$ is given by $M^{t}$. The pattern space $\Omega_{\times_{d}\sigma}$ associated to the $d$-fold product $\times_{d} \sigma$ of $\sigma$ with itself $d$ times is homeomorphic to the product $\prod_{i=1}^{d}\Omega_{\sigma}$. Via the K\"{u}nneth theorem, the induced action on $H^{*}(\prod_{i=1}^{d}\Omega_{\sigma},\mathbb{R})$ is the linear action $(M^{t})^{\otimes d}$ induced by $M^{t}$ on the $d$-fold tensor productd $H^{*}(\Omega_{\sigma},\mathbb{R})^{\otimes d}$.

The spectrum of the map $(M^{t})^{\otimes d}$ consists of all $d$-fold products of eigenvalues of $M^{t}$. Since $M^{t}$ is primitive it has a leading eigenvalue $\xi_1$ and the leading eigenvalue of $(M^{t})^{\otimes d}$ is $\xi_{1}^{d} = \nu_1$. The product substitution $\times_{d} \sigma$ is a pure dilation (coming from the dilation of $\sigma$) in dimension $d$, and an eigenvalue $\nu_{i}$ lies in the rapidly expanding subspace if $|\nu_{i}| > \nu_{1}^{ \frac{d-1}{d}} = \xi^{d-1}$ (from (\ref{eqn:RES})). It follows that if $\xi_{j}$ is an eigenvalue of $M^{t}$ with $|\xi_{j}| > 1$, then $\xi_{j} \xi_{1}^{d-1} > \xi_{1}^{d-1} = \nu_1^{\frac{d-1}{d}}$ and hence $\xi_{j} \xi_{1}^{d-1}$ (which is an eigenvalue for $(M^{t})^{\otimes d}$) corresponds to an eigenvector in the rapidly expanding subspace which corresponds to growth of ergodic integrals which are faster than the boundary effects.

\bibliographystyle{amsalpha}
\bibliography{biblio}

\end{document}